\documentclass[12pt]{article}
\usepackage{amsfonts}
\usepackage{amssymb}%
\usepackage{amsmath}%
\usepackage{latexsym}
\usepackage{amsthm}%
\usepackage{graphicx}%
\usepackage{epsfig}
\usepackage{subfigure,psfrag}%
\usepackage{clrscode}%
\usepackage{array}

\newcolumntype{D}{@{}>{\iffalse}l<{\fi}}

\newcommand{\sm}{\setminus}

\newcommand{\R}{\mathbb{R}}
\newcommand{\Rn}{\R^n}
\newcommand{\Rq}{\R^q}
\newcommand{\Rnplus}{\R^n_+}
\newcommand{\Z}{\mathbb{Z}}
\newcommand{\Zqplus}{\Z^q_+}
\newcommand{\qplus}{q^+}
\newcommand{\qmin}{q^-}

\newcommand{\twoq}{2^q}

\newcommand{\hal}{\hat{\al}}
\newcommand{\hm}{\widehat{m}}

\newcommand{\tm}{\widetilde{m}}
\newcommand{\bx}{\bar{x}}
\newcommand{\bv}{\bar{v}}
\newcommand{\bw}{\bar{w}}
\newcommand{\br}{\bar{r}}

\newcommand{\bm}{\bar{m}}

\newcommand{\rj}{r^j}
\newcommand{\sj}{s_j}

\newcommand{\sjs}{s_{\jstar}}
\newcommand{\rh}{r_h}
\newcommand{\fh}{f_h}
\newcommand{\xh}{x_h}

\newcommand{\istar}{i_*}
\newcommand{\jstar}{j_*}
\newcommand{\Qis}{Q_{\istar}}
\newcommand{\rijstar}{r^{\istar}_{\jstar}}
\newcommand{\bmijstar}{\bar{m}^{\istar}_{\jstar}}

\newcommand{\sse}{\subseteq}

\newcommand{\jinJ}{j \in J}
\newcommand{\vmi}{\bigvee^m_{i=1}}

\newcommand{\conv}{\mbox{\rm conv}}
\newcommand{\PD}{P_D}
\newcommand{\pde}{\PD^{=}}
\newcommand{\PI}{P_I}
\newcommand{\PL}{P_L}
\newcommand{\Pni}{P^{(n)}_i}
\newcommand{\Pnis}{P^{(n)}_{i_*}}
\newcommand{\Pnd}{P^{(n)}_D}
\newcommand{\convpd}{\conv\,\PD}
\newcommand{\convpi}{\conv\,\PI}

\newcommand{\convpnd}{\conv\,\Pnd}
\newcommand{\convpde}{\convpd^=}
\newcommand{\PnI}{P^{(n)}_I}
\newcommand{\convpnI}{\conv\,\PnI}
\newcommand{\pocta}{$P_{\text{octa}}$}

\newcommand{\cglp}{(CGLP)}
\newcommand{\cglpq}{\cglp$_Q$}

\newcommand{\al}{\alpha}

\newcommand{\js}{{j_*}}

\newcommand{\bal}{\bar{\al}}
\newcommand{\bu}{\bar{u}}

\newcommand{\half}{\frac{1}{2}}

\newcommand{\Qpi}{Q^+_i}
\newcommand{\Qmi}{Q^-_i}
\newcommand{\tqmi}{\widetilde{Q}^-_i}
\newcommand{\tqpi}{\widetilde{Q}^+_i}
\newcommand{\xk}{x_k}

\newcommand{\tkv}{\widetilde{K}^*(v)}
\newcommand{\tkvn}{\widetilde{K}^{*(n)}(v)}
\newcommand{\tal}{\widetilde{\al}}
\newcommand{\kvn}{K^{*(n)}(v)}

\newcommand{\rij}{r^i_j}
\newcommand{\mij}{m^i_j}
\newcommand{\hmoj}{\hm^1_j}
\newcommand{\hmtj}{\hm^2_j}
\newcommand{\roj}{r^{\,1}_j}
\newcommand{\broj}{\bar{r}^{\,1}_j}
\newcommand{\rtj}{r^2_j}
\newcommand{\brtj}{\bar{r}^{\,2}_j}
\newcommand{\moj}{m^1_j}
\newcommand{\mtj}{m^2_j}
\newcommand{\ceilroj}{\lceil \roj \rceil}
\newcommand{\ceilrtj}{\lceil \rtj \rceil}
\newcommand{\floorri}{\lfloor r^i \rfloor}
\newcommand{\ceilri}{\lceil r^i \rceil}

\newcommand{\floorbro}{\lfloor \bar{r}^1 \rfloor}
\newcommand{\ceilbro}{\lceil \bar{r}^1 \rceil}
\newcommand{\ceilbrt}{\lceil \bar{r}^2 \rceil}
\newcommand{\floorbrt}{\lfloor \bar{r}^2 \rfloor}

\newcommand{\Ak}{A^k}
\newcommand{\ako}{a^k_0}
\newcommand{\bko}{b^k_0}
\newcommand{\thk}{\theta^k}
\newcommand{\wk}{w_k}
\newcommand{\vk}{v_k}
\newcommand{\akj}{a^k_j}
\newcommand{\mkj}{m^k_j}
\newcommand{\alkj}{\al^k_j}
\newcommand{\balkj}{\bal^k_j}
\newcommand{\alj}{\al_j}
\newcommand{\ao}{a_0}
\newcommand{\bo}{b_0}
\newcommand{\bvk}{\bv_k}
\newcommand{\bvkp}{\bv^+_k}
\newcommand{\bwk}{\bw_k}
\newcommand{\bwkp}{\bw^+_k}

\newcommand{\fo}{f_1}
\newcommand{\ft}{f_2}
\newcommand{\ta}{\tilde{a}}
\newcommand{\tb}{\tilde{b}}

\newcommand{\almaxj}{\al^{\rm max}_j}
\newcommand{\alminj}{\al^{\rm min}_j}
\newcommand{\mjmax}{m^{\rm max}_j}
\newcommand{\mjmin}{m^{\rm min}_j}

\newcommand{\sevene}{(\ref{nonbasicdisjunction_integralnonbasics}$^=$)}

\topmargin by -\headsep \textheight 8.9in \oddsidemargin 0pt
\evensidemargin \oddsidemargin \marginparwidth 0.5in \textwidth
6.5in

\newtheorem{theorem}{Theorem}[section]
\newtheorem{corollary}[theorem]{Corollary}
\newtheorem{lemma}[theorem]{Lemma}
\newtheorem{proposition}[theorem]{Proposition}

\newenvironment{example}[1][Example]{\begin{trivlist}
\item[\hskip \labelsep {\bfseries #1}]}{\end{trivlist}}

\numberwithin{equation}{section}
\numberwithin{table}{section}

\begin{document}
\baselineskip=18pt      
\title{Intersection cuts from multiple rows: a disjunctive programming approach}
\date{April 2012}
\author{Egon Balas\thanks{Carnegie Mellon University, Tepper School of Business, Pittsburgh, PA 15213. Research supported by NSF Grant \#DMI-352885 and ONR contract \#N00014-03-1-0133} \and Andrea Qualizza\footnotemark[1]~\thanks{Currently at Amazon.com Research, Seattle, WA.}}

\maketitle

\abstract{
We address the issue of generating cutting planes for mixed integer programs from multiple rows of the simplex tableau with the tools of disjunctive programming. A cut from $q$ rows of the simplex tableau is an intersection cuts from a $q$-dimensional parametric cross-polytope, which can also be viewed as a disjunctive cut from a $2^q$-term disjunction. We define the disjunctive hull of the $q$-row problem, describe its relation to the integer hull, and show how to generate its facets. For the case of binary basic variables, we derive cuts from the stronger disjunctions whose terms are equations. We give cut strengthening procedures using the integrality of the nonbasic variables for both the integer and the binary case. Finally, we discuss some computational experiments.}

\section{Introduction: intersection cuts and disjunctive programming}\label{sec:intro}

In the last few years a considerable effort has been devoted to generating valid cuts for mixed integer programs from multiple rows of the simplex tableau, with a focus on cuts from two rows. This research was pioneered by the 2007 paper of Andersen, Louveaux, Weismantel and Wolsey \cite{ALWW}, followed by Borozan and Cornu\'ejols \cite{BorozanCornuejols}, Cornu\'ejols and Margot \cite{CM}, Dey and Wolsey \cite{DeyWolsey} and many others (\cite{BaBoCoMa1, Margot_code, DashDeyGunluk, DeyLodi}; for a recent survey see \cite{ConfortiCornuejolsZambelli}).

All of these papers view and derive the multiple-row cuts as intersection cuts, a concept introduced in \cite{BalasIntersection}, i.e. cuts obtained by intersecting the extreme rays of the cone defined by a basic linear programming solution with the boundary of a convex set whose interior contains no feasible integer point. Intersection cuts are equivalent to disjunctive cuts, and in this paper we apply the tools of disjunctive programming to the study of cuts from multiple rows of the simplex tableau. Two early versions of this paper were presented at the 2009 Spring Meeting of the AMS in San Francisco \cite{BalasDH1} and at the 20$^{th}$ ISMP in Chicago \cite{BalasDH2}.

The structure of our paper is as follows. In the remainder of this section we outline the connection of intersection cuts with disjunctive programming. In section~\ref{sec:integer_and_disjunctive_hulls} we introduce the concept of disjunctive hull associated with $q$ rows of the simplex tableau and examine the relation between the disjunctive hull and the integer hull. We then give a geometric interpretation of cuts from $q$ rows of the simplex tableau as cuts from a $q$-dimensional parametric cross-polytope (section~\ref{chapter2_subsection_cuts_from_parametric_crosspolytope}), followed by a theorem relating the facets of the disjunctive hull to those of the integer hull (section~\ref{chapter2_subsection_facets_dh_ih}). In section~\ref{chapter2_section_case_2_dimensions} we specialize these results to the case of $q=2$. The next section (\ref{chapter2_subsection_strengthening}) discusses the strengthening of our cuts when some of the nonbasic variables are integer-constrained. Section~\ref{chapter2_section_01_disjunctive_hull} deals with the 0-1 case, when the stronger disjunction whose terms are equations can be used to derive stronger cuts. Finally,
section~\ref{chapter2_section_computational_results} describes some computational experiments.

\vspace{12pt} \begin{center}{\bf *} \end{center} \vspace{12pt}

Suppose a Mixed Integer Program is given in the form of $q$ rows of the simplex tableau
\begin{equation}\label{eq2.1}
x=\bx + \sum_{\jinJ} \rj \sj, \quad x \in \Zqplus, \ s \in \Rnplus
\end{equation}
where $\bx$ is a basic feasible solution to LP, the linear
programming relaxation of a MIP, and we are interested in generating
an inequality that cuts off $\bx$ but no feasible integer point.

\begin{theorem}\label{th1.1} (Balas \cite{BalasIntersection}). Let $T \sse \R^{q}$ be a
closed convex set whose interior contains $\bx$ but no feasible integer point.
For $\jinJ$, let $\sj^* := \max \{ \sj : \bx + \rj \sj \in T\}$.
Then the inequality $\alpha s \ge 1$, where $\alpha_j =
\frac{1}{\sj^*}$, $\jinJ$, cuts off $\bx$ but no feasible
integer point.
\end{theorem}
The inequality $\al s \ge 1$ is known as an {\em intersection cut}.
\begin{figure}[ht!]
\setlength{\unitlength}{.9cm}
\begin{picture}(5,5)(0,0)
~~~~~~~~~~~~~~~~\scalebox{1}{\includegraphics{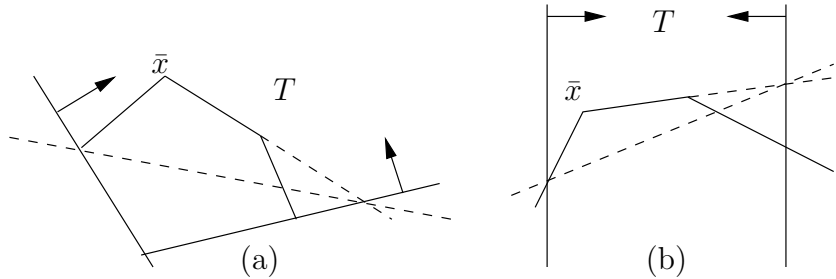}}
\put(-10.3,2.9){$\bx$}
\put(-8.5,2.5){$T$}
\put(-9,0){(a)}
\put(-2.9,3.5){$T$}
\put(-4.2,2.5){$\bx$}
\put(-3,0){(b)}
\end{picture}
\caption{Two intersection cuts}
\label{figure1}
\end{figure}
Theorem~\ref{th1.1} is illustrated by Figure~\ref{figure1}. In both cases (a) and (b) the convex set $T$ consists of the
intersection of two halfspaces, but in (b) the two
halfspaces are defined by hyperplanes parallel to one of the
coordinate axes, and so their intersection defines an infinite
strip. The intersection cut from this latter set $T$ is the Gomory
Mixed Integer cut (GMI) \cite{Gomory1960}.

This particular class of intersection cuts, the GMI cuts, has played
a crucial role in making mixed integer programs practically
solvable. These cuts are derived from a convex set of the form
$\lfloor \bx_i \rfloor \le x_i \le \lceil \bx_i \rceil$, where $x_i
= \bx_i + \sum_{\jinJ} r_j^i \sj$ is one of the rows of an optimal
simplex tableau and $\lfloor \bx_i \rfloor < \bx_i < \lceil \bx_i \rceil$. More generally, cuts obtained from a convex set of
the form $\pi_0 \le \pi x \le \pi_0+1$, where $(\pi,\pi_0)$ is an
integer vector with $\text{gcd}(\pi) = 1$, are known in the literature as split cuts \cite{CKS}. It is
then natural to ask the question whether intersection cuts derived
simultaneously from several rows of a simplex tableau have some
properties that distinguish them from split cuts. It was this
question that has led to the investigation of intersection cuts from
maximal lattice-free convex sets by \cite{ALWW,BorozanCornuejols} and others.

We propose a different approach to the same problem, which promises
some computational advantages. The approach is that of Disjunctive
Programming, a natural outgrowth of the study of intersection cuts.
To see the connection, consider an intersection cut from a
polyhedral set with the required properties, of the form $T := \{ x
: d^i x \le d^i_0, \ i=1,\ldots,m\}$. Clearly, the requirement that
int\,$T$ should contain no feasible integer point, can be rephrased
as the requirement that every feasible integer point should satisfy
at least one of the weak complements of the inequalities defining
$T$, i.e. should satisfy the disjunction
\begin{equation}\label{eq2.2}
\bigvee^m_{i=1} (d^i x \ge d^i_0).
\end{equation}
Therefore an intersection cut from $T$ can be viewed as a disjunctive cut
from (\ref{eq2.2}). While these two cuts are essentially the same, the
disjunctive point of view opens up new perspectives. Thus, suppose
that in addition to (\ref{eq2.2}), all feasible solutions have to
satisfy the inequalities $Ax \ge b$. Then one way to proceed is to
generate all valid cutting planes from (\ref{eq2.2}) and append
these to $Ax \ge b$. The resulting system will be
$$P := \left\{ x \in \Rn : \left(Ax \ge b\right) \cap \mbox{ conv\,} \left(
\bigvee^m_{i=1} \left(d^i x \ge d^i_0\right)\right)\right\}.
$$
But another way to proceed is to introduce $Ax \ge b$ into each term
of the disjunction (\ref{eq2.2}), i.e. replace (\ref{eq2.2}) with
\begin{equation}\label{eq2.3}
\vmi \left( \begin{array}{c} Ax \ge b \\ d^i x \ge
d^i_0\end{array}\right),
\end{equation}
and take the convex hull of this union of polyhedra:
$$Q := \mbox{ conv\,} \left( \vmi \left( \begin{array}{c} Ax \ge b \\ d^i x \ge
d^i_0\end{array}\right) \right)
$$
Now it is not hard to see that $Q \sse P$, and in fact $Q$ is in
most cases a much tighter constraint set than $P$. We illustrate the
difference on a 2-term disjunction. Given an arbitrary Mixed Integer
Program, let $(\pi,\pi_0)$ be an integer vector with a component
$\pi_j$ for every integer-constrained variable. Then the disjunctive
cut derived from 
\begin{equation}\label{eq2.4}
\pi x \le \pi_0 \ \vee \ \pi x \ge \pi_0 + 1 
\end{equation}
is equivalent to the intersection cut derived from the convex set
$$\pi_0 \le \pi x \le \pi_0 +1,
$$
illustrated in Figure~\ref{figure1}. On the other hand, the
disjunction
\begin{equation}\label{eq2.5}
\left(\begin{array}{rcl} Ax & \ge & b \\ \pi x & \le & \pi_0
\end{array}\right) \quad \vee \quad \left(\begin{array}{rcl} Ax & \ge & b \\ \pi x & \ge &
\pi_0 +1
\end{array}\right)
\end{equation}
gives rise to an entire family of cuts, whose members are determined
by the multipliers $u$, $v$ associated with $Ax \ge b$ in the two
terms of this more general disjunction
\begin{equation}\label{eq2.6}
(\pi - uA) x \le \pi_0 - ub \ \vee \ (\pi + vA) x \ge \pi_0 + vb
+1
\end{equation}
Cuts derived from a disjunction of the form (\ref{eq2.4}) are
called split cuts, a term that reflects the fact that (\ref{eq2.4})
splits the space into two disjoint half-spaces. Cook, Kannan and
Schrijver \cite{CKS} who coined this term also extended it to the much
larger family of cuts derived from disjunctions of the form
(\ref{eq2.6}).

Disjunctive sets of the form (\ref{eq2.3}) or (\ref{eq2.5})
represent unions of polyhedra, and the study of optimization over
unions of polyhedra is known as Disjunctive Programming. Its two
basic results are a compact representation of the convex hull of a
union of polyhedra in a higher dimensional space, and the sequential
convexifiability of facial disjunctive sets \cite{BalasDisjunctiveProgramming,Balas}. The
application of disjunctive programming to mixed 0-1 programs has
become known as the lift-and-project method \cite{BalasCeriaCornuejols}.
Here we apply this approach to the study of intersection cuts from
multiple rows of the simplex tableau.

\section{Integer and disjunctive hulls}\label{sec:integer_and_disjunctive_hulls}

Consider again a system defined by $q$ rows of the simplex tableau,
this time without the integrality constraints:
\begin{equation}\label{eq3.1}
x = f + \sum_{\jinJ} \rj \sj, \ \sj \ge 0, \ j \in J,
\end{equation}
where $f$, $\rj \in \Rq$, $\jinJ := \{ 1,\ldots,n\}$, and
assume $0 < f_i < 1$, $i \in Q := \{ 1,\ldots,q\}$.
This assumption can be made without loss of generality since setting
$x_i^\prime = x_i - \lfloor f_i\rfloor$ and $f_i^\prime = f_i - \lfloor
f_i\rfloor,\ i\in\ Q$, we have that $x_i^\prime,f_i^\prime,\ i\in Q$ satisfy the assumption.
The set
\begin{equation}\label{eq3.2}
P_L := \{ (x,s) \in \Rq \times \Rn : (x,s) \mbox{ satisfies
(\ref{eq3.1})}\}
\end{equation}
is the polyhedral cone with apex at $(x,s)=(f,0)$ defined by the
constraints that are tight for this particular basic solution.
Imposing the integrality constraints on the basic components we get
the mixed integer set
\begin{equation}\label{eq3.3}
P_I := \{ (x,s) \in P_L : x_i \mbox{ integer}, i \in Q\},
\end{equation}
whose convex
hull, conv\,$P_I$, is Gomory's corner polyhedron \cite{Gomory}, or the
{\em integer hull} of the MIP over the cone $P_L$. The main objective of
the papers mentioned in the introduction was to study the structure
of $P_I$ for small $q$, with a view of characterizing the facets of
conv\,$P_I$ and minimal valid inequalities for $P_I$.

Consider now the following disjunctive relaxation of $P_I$,
obtained by replacing the integrality constraints on $x_i$ with the simple
disjunctions $x_i \le 0 \ \vee \ x_i \ge 1$, $i \in Q$:
\begin{equation}\label{eq3.4}
\PD := \{ (x,s) \in \PL : x_i \le 0 \ \vee \ x_i \ge 1, \ i \in
Q\}.
\end{equation}

Like $\PI$, $\PD$ is a nonconvex set. Its convex hull, $\conv\,\PD$,
which we call the simple {\em disjunctive hull}, is a weaker relaxation of $\PI$ than $\conv\,\PI$, i.e. $\conv\,\PD
\supseteq \conv \PI$, but it is easier to handle, since it is the
convex hull of the union of $2^q$ polyhedra. Thus one can apply
disjunctive programming and lift-and-project techniques to generate
facets of $\conv\,\PD$ at a computational cost that for small $q$
seems acceptable. In this context, the crucial question is of
course, how much weaker is the relaxation $\conv\,\PD$ than
$\conv\,\PI$? We will pose this question in a more specific form
that will enable us to give it a practically useful answer: when is
it that a facet defining inequality for $\convpd$ is also facet
defining for $\convpi$? In other words, which facets of the (simple)
disjunctive hull are also facets of the integer hull?
Before addressing this question, however, we will take a side-step,
by introducing a third kind of hull. If we strengthen the
disjunctive relaxation of $\PI$ by replacing the inequalities in the
disjunctions $x_i \le 0 \ \vee \ x_i \ge 1$, $i\in Q$, with
equations, we get the set 
\begin{equation}\label{eq3.5}
\pde := \{ (x,s) \in \PL : x_i = 0 \ \vee \ x_i =1, \ i \in
Q\},
\end{equation}
whose convex hull, $\convpde$, we call {\em the 0-1 disjunctive hull}. For a
general mixed integer program, the 0-1 Disjunctive Hull is not a
valid relaxation, in that it may cut off nonbinary feasible integer
points. Indeed, we have
$$\convpd \supseteq \convpi \supseteq \convpde,$$
where both inclusions are strict and are valid in the context of
mixed integer 0-1 programs only, since all the non-0-1 integer points
that it cuts off are infeasible.
Hence $\convpde$ is equivalent to the convex
hull of $\PI \cap \{x : x_i \le 1, \ i \in Q\}$, or the
integer hull of $\PI$ reinforced with the bounds on the $x_i$.
However, as we will see later on, finding facets of $\convpde$
requires roughly the same computational effort as finding facets of
$\convpd$.

The upshot of this is that for the important class of mixed integer
0-1 programs, all facet defining inequalities of $\convpde$ are
facet defining for the integer hull. Furthermore, from the
sequential convexification theorem of disjunctive programming, all
such inequalities are of split rank $\le q$, i.e. they can be
obtained by applying a split cut generating procedure at most $q$
times recursively.

The set $\PD$ of (\ref{eq3.4}) is the collection of those points
$(x,s) \in \Rq \times \Rn$ satisfying (\ref{eq3.1}) and $x_i \le 0 \
\vee \ x_i \ge 1$, $i \in Q$. Put in disjunctive normal form, this
last constraint set becomes
\begin{equation}\label{eq3.6}
\left( \begin{array}{c} x_1 \le 0 \\ x_2 \le 0 \\ \vdots
\\ x_q \le 0\end{array} \right) \ \vee \ \left( \begin{array}{c} x_1 \ge 1 \\ x_2 \le 0 \\ \vdots
\\ x_q \le 0\end{array} \right) \ \vee \ \cdots \ \vee \ \left( \begin{array}{c} x_1 \ge 1 \\ x_2 \ge 1 \\ \vdots
\\ x_q \ge 1\end{array} \right)
\end{equation}

Each term of (\ref{eq3.6}) defines an orthant-cone with apex at a
vertex of the $q$-dimensional unit cube. These $2^q$ orthant-cones
are illustrated for $q=2$ in Figure~\ref{figure3}.
\begin{figure}[htbp]
\begin{picture}(10,200)(-110,0)\setlength{\unitlength}{.15cm}
\put(6,31){\line(1,1){13}}
\put(6,33){\line(1,1){11}}
\put(6,35){\line(1,1){9}}
\put(6,37){\line(1,1){7}}
\put(6,39){\line(1,1){5}}
\put(8,31){\line(1,1){11}}
\put(10,31){\line(1,1){9}}
\put(12,31){\line(1,1){7}}
\put(20,30){\line(-1,0){15}}
\put(20,30){\line(0,1){15}}
\put(20,30){\circle*{1}}
\put(18,27){0,1}
\put(20,15){\line(-1,0){15}}
\put(20,15){\line(0,-1){15}}
\put(20,15){\circle*{1}}
\put(18,16.5){0,0}
\put(6,1){\line(1,1){13}}
\put(6,3){\line(1,1){11}}
\put(6,5){\line(1,1){9}}
\put(6,7){\line(1,1){7}}
\put(6,9){\line(1,1){5}}
\put(8,1){\line(1,1){11}}
\put(10,1){\line(1,1){9}}
\put(12,1){\line(1,1){7}}
\put(35,30){\line(1,0){15}}
\put(35,30){\line(0,1){15}}
\put(35,30){\circle*{1}}
\put(33,27){1,1}
\put(36,31){\line(1,1){13}}
\put(36,33){\line(1,1){11}}
\put(36,35){\line(1,1){9}}
\put(36,37){\line(1,1){7}}
\put(36,39){\line(1,1){5}}
\put(38,31){\line(1,1){11}}
\put(40,31){\line(1,1){9}}
\put(42,31){\line(1,1){7}}
\put(35,15){\line(1,0){15}}
\put(35,15){\line(0,-1){15}}
\put(35,15){\circle*{1}}
\put(33,16.5){1,0}
\put(36,1){\line(1,1){13}}
\put(36,3){\line(1,1){11}}
\put(36,5){\line(1,1){9}}
\put(36,7){\line(1,1){7}}
\put(36,9){\line(1,1){5}}
\put(38,1){\line(1,1){11}}
\put(40,1){\line(1,1){9}}
\put(42,1){\line(1,1){7}}
\end{picture}
\caption{Orthant-cones for the case $q=2$}
\label{figure3}
\end{figure}
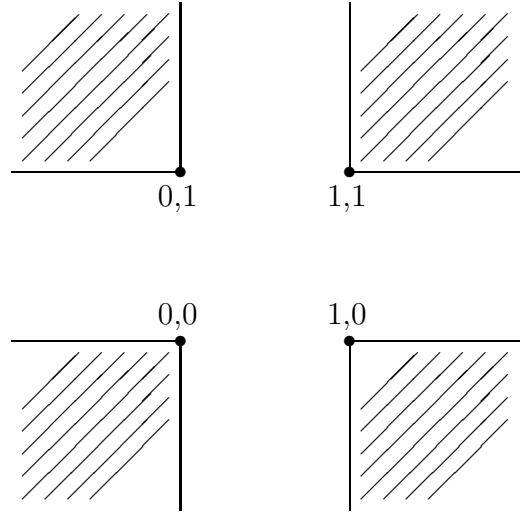

Using (\ref{eq3.1}) to eliminate the $x$-components and denoting by
$r^i$ the $i$-th row of the $q \times n$ matrix $R =
(\rj)^n_{j=1}\,$, (\ref{eq3.6}) can be represented in $\Rn$ as $s
\ge 0$ and
\begin{equation}\label{eq3.7}
\left( \begin{array}{rcl} -r^1s & \ge & f_1 \\ -r^2 s & \ge & f_2
\\ & \vdots & \\ -r^q s & \ge & f_q \end{array}\right) \ \vee \ \left( \begin{array}{rcl} r^1s & \ge & 1- f_1 \\ -r^2 s & \ge & f_2
\\ & \vdots & \\ -r^q s & \ge & f_q \end{array}\right) \ \vee \
\cdots \ \vee \ \left( \begin{array}{rcl} r^1s & \ge & 1-f_1 \\ r^2
s & \ge & 1- f_2
\\ & \vdots & \\ r^q s & \ge & 1- f_q \end{array}\right)
\end{equation}

If $\Pni \sse \Rn$ denotes the polyhedron defined by the $i$-th term
of this disjunction plus the constraints $s \ge 0$, then $\PD$ can
be defined in $n$-space as $\Pnd = \cup^t_{i=1} \Pni$ where $t =
2^q$. Furthermore, we have the following:

\begin{theorem}\label{th5.1}
$\convpnd$ is the set of those $s \in \Rn$ satisfying $s \ge 0$ and
all the inequalities $\alpha s \ge 1$ whose coefficient vectors
$\alpha \in \Rn$ satisfy the system

{\scriptsize
\begin{equation}\label{eq3.8}
\begin{array}{r}
\alpha + r^1 u_{11} + \cdots + r^q u_{1q} ~~~~~~~~~~~~~~~~~~~~~~~~~~~~~~~~~~~~~~~~~~~~~~~~~~~~~~~~~~~~~~~~~~~~~~~~ \ge 0\\[6pt]
\alpha ~~~~~~~~~~~~~~~~~~~~~~~~~~~~~~~- r^1 u_{21} + \cdots + r^q
u_{2q}
~~~~~~~~~~~~~~~~~~~~~~~~~~~~~~~~~~~~~~~~~ \ge 0\\[6pt]
\vdots~~~~~~~~~~~~~~~~~~~~~~~~~~~~~~~~~~~~~~~~~~~~~~~~~~~~~~~~~~~~~~~~\ddots ~~~~~~~~~~~~~~~~~~~~~~~~~~~~~~~~~\vdots ~~\\[6pt]
\alpha
~~~~~~~~~~~~~~~~~~~~~~~~~~~~~~~~~~~~~~~~~~~~~~~~~~~~~~~~~~~~~~~~~~~~~~~~~-
r^1u_{t1} -
\cdots - r^qu_{tq} \ge 0\\[6pt]
f_1 u_{11} + \cdots + f_q u_{1q} ~~~~~~~~~~~~~~~~~~~~~~~~~~~~~~~~~~~~~~~~~~~~~~~~~~~~~~~~~~~~~~~~~~~~~~~~ \ge 1\\[6pt]
~~~~~~~~~~~~~~~~~~~~~~~~~~~(1\!-\!f_1)u_{21} + \cdots + f_q
u_{2q}~~~~~~~~~~~~~~~~~~~~~~~~~~~~~~~~~~~~~~~~~ \ge 1\\[6pt]
\ddots ~~~~~~~~~~~~~~~~~~~~~~~~~~~~~~~~~~~~~~ \vdots~~\\[6pt]
  (1\!-\!f_1) u_{t1} + \cdots
+ (1\!-\!f_q) u_{tq} \ge 1
\end{array}
\end{equation}}
for some $u_{ik} \ge 0$, $i=1,\ldots,t=2^q, k=1,\ldots,q$.
\end{theorem}

\begin{proof}
Applying the basic theorem of Disjunctive Programming
to $\convpnd$ we introduce auxiliary variables $s^i \in \Rn$, $z_i
\in \R$, $i=1,\ldots,t=2^q$, and obtain the higher-dimensional
representation

{\small 
\begin{equation}\label{eq3.9}
\begin{array}{rrcrrrcrcl}
s~~-s^1    & -s^2 & \ldots & -s^t &            &                 &&& =   &
0\\[6pt]
-r^1 s^1 &      &        &      & -f_1 z_1    &                 &&& \ge & 0\\[6pt]
-r^2 s^1 &      &        &      & -f_2 z_1 &                 &&& \ge & 0\\[6pt]
\vdots~~~&      &        &      & \vdots ~~&                 &&& \vdots \\[6pt]
-r^q s^1 &      &        &      & -f_q z_1   &                 &&& \ge & 0\\[6pt]
         & r^1s^2&       &      &            & -(1\!-\!f_1)z_2 &&& \ge & 0\\[6pt]
         & -r^2 s^2&     &      &            & -f_2 z_2        &&& \ge & 0\\[6pt]
         &  \vdots~~ &     &      &            & \vdots  ~~        &&&
         \vdots \\[6pt]
         & -r^q s^2&     &      &            & -f_q z_2        &&& \ge & 0\\[6pt]
         &      & \ddots &      &            &                 & \ddots && \vdots\\[6pt]
         &      &        & -r^1 s^t&         &                 &        &  -(1\!-\!f_1) z_t & \ge & 0\\[6pt]
         &      &        & -r^2 s^t&         &                 &        &  -(1\!-\!f_2) z_t & \ge & 0\\[6pt]
         &      &        & \vdots~~&         &                 &        &  \vdots ~~~~& \vdots  \\[6pt]
         &      &        & -r^q s^t&         &                 &        &  -(1\!-\!f_q) z_t & \ge &
         0\\[6pt]
         &&&& z_1 & + z_2 & + \cdots & + z_t & = & 1\\[6pt]
\multicolumn{10}{c}{s^i \ge 0, \ i=1,\ldots,t; \ \ z_i \ge 0, \
i=1,\ldots,t}
\end{array}
\end{equation}}

Projecting this system onto the $s$-space with multipliers $\alpha$;
$u_{11},\ldots,u_{1q}$; $u_{21},\ldots,u_{2q}$; $\ldots$;
$u_{t1},\ldots,u_{tq}$, we obtain

{\scriptsize \begin{equation}\label{eq3.10}
\begin{array}{rlrrrrrrrrrcl}
\alpha & +~~r^1u_{11} & + \cdots + & r^qu_{1q} & &&&& \ge & 0\\[6pt]
\vdots & & & & \ddots &&&& \vdots~ &\\[6pt]
\alpha & & & &       & -r^1u_{t1} ~~~~~& - \cdots - & -r^qu_{tq} &
\ge & 0\\[6pt]
& -\beta + f_1 u_{11} & + \cdots + & f_q u_{1q} & & & & & \ge & 0\\[6pt]
& \vdots~~~~~~~~ & & & \ddots &&&& \vdots~ &\\[6pt]
& - \beta &&&& +(1\!-\!f_1)u_tf_1 & + \cdots + & (1\!-\!f_q)u_{tq} &
\ge & 0\\[6pt]
 \multicolumn{9}{c}{u_{ik} \ge 0, \ i=1,\ldots,t,\ k=1,\ldots,q}
\end{array}
\end{equation}}

Applying the normalization $\beta=1$ (clearly $\beta = -1$ does not
yield any cuts since it makes (\ref{eq3.10}) unbounded) we obtain the
representation given in the theorem.
\end{proof}

In order to restate the system (\ref{eq3.8}) in a more concise form,
for each $i \in \{ 1, \ldots, t\}$ we partition the index set $Q :=
\{1,\ldots,q\}$ into
\begin{tabbing}
~~~~~~\= $Q^+_i := \{ k \in Q : u_{ik}$ has coefficient vector $r_k
\}$\\
\> $Q^-_i := \{ k \in Q : u_{ik}$ has coefficient vector $-r_k \}$,
\end{tabbing}
with $Q^+_i \cup Q^-_i = Q$, $i=1,\ldots,t=2^k$. Then (\ref{eq3.8})
can be restated as
$$\begin{array}{rclclcll}
\displaystyle\alpha &\!\!\!+\!\!\!& \displaystyle\sum \left(r^k u_{ik} : k \in
Q^+_i\right) &\!\!\!-\!\!\!\!& \displaystyle\sum\left(r^k u_{ik} : k
\in Q^-_i\right) & \ge & 0\\[12pt]
&& \displaystyle\sum(f_k u_{ik} : k \in Q^+_i) &\!\!\!\!\!\!\!+\!\!\!\!\!\!\!&
\displaystyle\sum( (1-f_k) u_{ik} : k \in
Q^-_i) & \ge & 1, & i=1,\ldots,t\\[12pt]
\multicolumn{8}{c}{\displaystyle u_{ik} \ge 0, \ i=1,\ldots,t=2^q, \
k=1,\ldots,q}
\end{array}\eqno{\mbox{(\ref{eq3.8}$'$)}}\label{eq3.8'}
$$

The system (\ref{eq3.8}$'$) has several interesting properties
described in the next few propositions.

\begin{proposition}\label{pr2}
For any $p \in \Rn$, $p > 0$, all optimal basic solutions to the cut
generating linear program
$$\min \{ p \alpha : (\alpha,u) \mbox{ satisfies (\ref{eq3.8}$'$)}\}
\eqno{\mbox{\rm (CGLP)$_Q$}}
$$
are of the form
\begin{equation}\label{eq3.11}
\alpha_j = \max \{ \alpha^1_j,\ldots,\alpha^t_j\},
\end{equation}
where
\begin{equation}\label{eq3.12}
\alpha^i_j := - \sum (r_j^k u_{ik} : k \in Q^+_i) + \sum (r_j^k
u_{ik} : k \in Q^-_i),
\end{equation}
$i=1,\ldots,t=2^q$, with the $u_{ik}$ satisfying (\ref{eq3.8}$\,'$).
\end{proposition}

\begin{proof}
The constraints of (\ref{eq3.8}$'$) require
$$\alpha_j \ge \alpha^i_j, \quad i=1,\ldots,t, \ j=1,\ldots,n
$$
Suppose there is an optimal solution to \cglpq\ such that
$\alpha_{j_*} > \max \{ \alpha^i_{j_*}: i = 1,\ldots,t\}$ for some
$j_* \in \{ 1,\ldots,n\}$. Then setting $\al_{\js}$ equal to the
maximum on the righthand side, and leaving $\al_j$ unchanged for all
$j \ne j_*$ yields a better solution, contrary to the assumption.
\end{proof}

\begin{proposition}\label{pr3}
In any valid inequality $\al s \ge 1$ for $\convpnd$, $\al_j \ge 0$,
$j = 1,\ldots,n$.
\end{proposition}

\begin{proof}
From (\ref{eq3.11}), $\al_j \ge \al^i_j$ for all $i=1,\ldots,\twoq$,
and in view of the presence of all sign patterns of $r_j^k u_{ik}$ in  the expressions
(\ref{eq3.12}), there is always an index $i \in \{ i,\ldots,2^q\}$ with $\al^i_j \ge 0$.
\end{proof}

\begin{proposition}\label{pr4}
For any basic solution $(\al,u)$ to \cglpq\ that satisfies as strict
inequality some of the nonhomogeneous constraints of (\ref{eq3.8}$\,'$),
there exists a basic solution $(\bal,u)$, with $\bal = \al$, that
satisfies at equality all the nonhomogeneous constraints of \cglpq.
\end{proposition}

\begin{proof}
Let $(\al,u)$ be a basic solution to \cglpq\ that satisfies as
strict inequality some of the nonhomogeneous constraints of
(\ref{eq3.8}$'$). W.l.o.g., assume that
$$f_1u_{11} + \cdots + f_q u_{1q} - \theta =1
$$
is one of those constraints with the surplus variable $\theta$
positive in the solution $(\al,u)$. We will show that there exists a
solution $(\bal,\bu)$, with $\bal = \al$ and $\bu_{ik} = u_{ik}$ for
all $i \ne 1$ and all $k$, such that
$$f_1 \bu_{11} + \cdots + f_q \bu_{1q} = 1.
$$
Applying this argument recursively then proves the Proposition.

Fix all variables of \cglpq\ except for $u_{11},\ldots,u_{1q}$, at
their values in the current solution. The fixing includes all the
surplus variables except those in the  $n+1$ rows containing
$u_{11},\ldots,u_{1q}$. This leaves the following constraint set in
the free variables:
\begin{equation}\label{eq3.13}
\begin{array}{rcll}
-r_j^1 u_{11} - \cdots - r_j^q u_{1q} + t_j & = & \bal_j &~~~~~
j=1,\ldots,n\\[6pt]
f_1 u_{11} + \cdots + f_q u_{1q} - \theta &=& 1\\[6pt]
\multicolumn{4}{c}{u_{11}, \ldots, u_{1q} \ge 0, \ t_j \ge 0, \
j=1,\ldots,n, \ \theta \ge 0}
\end{array}
\end{equation}
Here $\theta, t_j$ represent the surplus variables of the respective
constraints. We claim that this system has a solution with
$\theta=0$. To see this, consider the linear program
$$\min \{ \theta : u_{ik}, t_j \mbox{ and } \theta \mbox{ satisfy
(\ref{eq3.13})}\}
$$
and its dual,
$$\max \lambda_0 + \sum^n_{j=1} \bal_j
\lambda_j~~~~~~~~~~~~~~~~~~~~~~~~~~~~
$$
subject to
$$\begin{array}{rcll}
f_1 \lambda_0 - \sum\limits^n_{j=1} r_j^1 \lambda_j &\le&0\\[6pt]
\vdots~~~~~~~~~~~~~ &\vdots& \\[6pt]
f_q \lambda_0 - \sum\limits^n_{j=1} r_j^q \lambda_j&\le&0\\[6pt]
-\lambda_0 ~~~~~~~~~~~~~~&\le& 1\\[6pt]
\lambda_j &\le&0, & j=1,\ldots,n
\end{array}
$$
Since $\bal_j \ge 0$, $j=1,\ldots,n$, it is not hard to see that the dual linear program has an optimal
solution $\lambda_0=0$, $\lambda_j=0$, $j=1,\ldots,n$ and hence the
primal has an optimal solution with $\theta=0$.
\end{proof}

The obvious and important consequence of
Proposition~\ref{pr4} is that for all practical purposes we can
replace all $\twoq$ nonhomogeneous inequalities in the constraint
set (\ref{eq3.8}$'$) of \cglpq\ with equations. In view of
Proposition~\ref{pr2}, it then follows that we may restrict our
attention to basic feasible solutions that satisfy at equality
$n+\twoq$ out of the $n \times \twoq + \twoq$ inequalities of
(\ref{eq3.8}$'$) other than the nonnegativity constraints.

At this point we introduce the characterization of $\convpde$, the
0-1 disjunctive hull defined by (\ref{eq3.5}), closely related to
that of $\convpd$. Just as in the case of $\PD$, we denote by
$P_D^{=(n)}$ the union of polyhedra in $\Rn$ representing the disjunction
(\ref{eq3.7}) in which all the inequalities have been replaced by equations. The following Theorem is the analog of Theorem~\ref{th5.1} for this case.

\begin{theorem}\label{th4.5}
$\convpd^{=(n)}$ is the set of those $s \in \R^n$ satisfying $s \ge
0$ and all inequalities $\al s \ge \beta$ whose coefficients satisfy
the system
\begin{equation}\label{eq3.14}
\begin{array}{lllcl}
\al  + r^1 u_{11} + \cdots + r^1 u_{1q} &&& \ge & 0\\
~\vdots & \ddots & & \vdots & \\
\al & &-r^1 u_{t1} - \cdots - r^q u_{tq} & \ge & 0\\
~~-\!\beta + f_1 u_{11} + \cdots + f_q u_{1q} &&& = & 0\\
~\vdots & \ddots & & \vdots \\
~~-\!\beta & & + (1\!\!-\!\!f_1) u_{t1} + \cdots + (1\!\!-\!\!f_q)u_{tq} & = & 0
\end{array}
\end{equation}
for some $u_{ik}$, $i=1,\ldots,t=\twoq$, $k=1,\ldots,q$.
\end{theorem}

\begin{proof}
The proof of Theorem~\ref{th5.1} goes through with the following
modifications. Since the inequalities in the disjunction
(\ref{eq3.7}) are all replaced with equations, the
inequalities in the system (\ref{eq3.9}), other than the
nonnegativity constraints, also become equations. As a consequence,
the variables $u_{ik}$ of the projected system (\ref{eq3.10}) become
unrestricted in sign. The remaining difference between (\ref{eq3.14})
and (\ref{eq3.8}) is the fact that in (\ref{eq3.14}) the last $\twoq$
constraints are equations rather than inequalities. This is due to
the fact that Proposition~\ref{pr4} applies here too. In other
words, if we denote by (\ref{eq3.14}$''$) the system obtained from
(\ref{eq3.14}) by replacing the equations containing $\beta$ with
inequalities $\ge$, then for any basic solution $(\al,u)$ to
(CGLP)$_Q$ that satisfies as strict inequalities some of the
constraints (\ref{eq3.14}$''$) containing $\beta$, there exists a
basic solution $(\bal,u)$, with $\bal = \al$, that satisfies at
equality all the constraints containing $\beta$. The proof is
essentially the same as that of Proposition~\ref{pr4}.

Thus the two basic differences between the systems describing
$\convpd^{(n)}$ and $\convpd^{=(n)}$ are that (a)~the latter also
contains inequalities of the form $\al x \le 1$ (corresponding to $\beta < 0$), and (b) the
coefficients $\al_j$ of the latter can be of any sign.
\end{proof}

We now return to the simple disjunctive hull, $\convpd$, and describe its
vertices.

\begin{proposition}\label{pr5}
Every vertex of $\convpnd$ is a vertex of some $\Pni$, $i \in
\{1,\ldots,\twoq\}$.
\end{proposition}

\begin{proof}
Let $v$ be a vertex of $\convpnd$. If $v \in \Pni$ for some $i \in
\{1,\ldots,t=2^q\}$, then $v$ must be a vertex of $\Pni$, or else it
could be expressed as a convex combination of points in $\Pni$,
hence of $\Pnd$. On the other hand, if $v \not\in \cup \Pni$ but $v
\in \conv \Pni $, then $v$ is a convex combination of points in
$\cup^t_{i=1} \Pni$, hence of $\convpnd$, a contradiction.
\end{proof}

Next we describe the vertices of $\Pni$, $i\in \{1,\ldots,\twoq\}$.
We will call a vertex of $\convpnd$ (of $\Pni$) integer if it
defines an integer $x$ through (\ref{eq3.1}); in other words if $f_i
+ r^i s$ is integer for $i = 1,\ldots,q$. All other vertices will be
called fractional.

For any particular $i_* \in \{1,\ldots,\twoq\}$,
$$\Pnis := \{ s \in \Rnplus : \rh s \le - \fh, \ h \in \Qis, \ \rh s
\ge 1-\fh, \ h \in Q \sm \Qis \}
$$
where $(\Qis, Q \sm \Qis)$ is the partition of $Q$ that defines
$\istar$.

\begin{proposition}\label{pr4.6}
$\Pnis$ can have three kinds of vertices, distinguished by the
corresponding $x$-vectors that belong to one of these types:
\begin{itemize}
\item[(a)] 0-1 vertices: $x_h = 0$, $h \in Q_{i_*}$ and $\xh=1$, $h
\in Q \sm \Qis$.
\item[(b)] non-binary integer vertices: $\xh \in \Z_-$, $h \in
\Qis$, $\xh \in \Z_+$, $h \in Q \sm \Qis$ (here $\Z_-$ and $\Z_+$
stand for the nonpositive and nonnegative integers respectively).
\item[(c)] fractional vertices: $\xh \le 0$, $h \in \Qis$, $\xh \ge
1$, $h \in Q \sm \Qis$, with at least one inequality strict.
\end{itemize}
\end{proposition}

\begin{proof}
The three cases become exhaustive if the following fourth one is
added: (d)~fractional vertices with $0 < \xh < 1$ for some $h \in
Q$. But this case clearly violates at least one of the constraints
defining $\Pnis$.
\end{proof}

Note that $\Pnis$ can have several distinct vertices with the same
associated $x$-vector, corresponding to basic solutions with the same $x$-component. Note also that if a component $\xh$ of a
vertex is fractional, then $\xh < 0$ or $\xh > 1$.

The next theorem characterizes the facets of the simple disjunctive hull.

\begin{theorem}\label{th1} 
The inequality $\bal s \ge 1$ defines a facet of $\convpnd$ if and
only if there exists an objective function of the linear program \cglpq\ of Proposition~\ref{pr2} with $p > 0$ such that all optimal solutions
$(\al,u)$ have $\al = \bal$.
\end{theorem}

\noindent {\em Proof outline.} This is a special case of Theorem~4.6
of \cite{Balas}. The inequality $\bal x \ge 1$ defines a facet of
$\convpnd$ if and only if $\bal$ is a vertex of the polar of
$\convpnd$, which is the projection of (\ref{eq3.8}) onto the
$\al$-space. But $\bal$ is a vertex of this polar if and only if
there exists an objective function vector $p > 0$ such that $p \al$
attains its unique minimum at $\bal$.\hfill$\Box$

If the system (\ref{eq3.4}) defining $\PL$ is of full row rank $q$,
then the dimension of $\convpd$ is $n$, since there are $q+n$
variables and $q$ independent equations. The dimension of
$\convpd^{(n)}$ is also $n$, so the facets of $\convpd^{(n)}$ are of
dimension $n-1$.

From a computational standpoint, the most important feature of
\cglpq\ is that the facets of the $n$-dimensional $\convpnd$ can be
generated by solving a smaller CGLP in a subspace of at most $t=2^q$
variables $\sj$, and lifting the resulting inequality into the full
space. The idea of generating cuts in a subspace of the original
higher dimensional cut generating linear program and then lifting
them to the full space goes back to \cite{BalasCeriaCornuejols, BalasPerregaard}, where
lift-and-project cuts were generated from a 2-term disjunction by
working in the subspace of the fractional variables of the LP
solution. Here we are working with a $\twoq$-term disjunction, and
are considering a different subspace, suggested by the structure of
the problem at hand, but the lifting procedure is essentially the
same as the one used in \cite{BalasCeriaCornuejols,BalasCeriaCornuejols2}.

Since our cuts are derived from a disjunction with $2^q$ terms, if
we want to create a subproblem in which all terms are represented,
we need $2^q$ out of the $n$ variables $\al_j$ of our \cglpq.
Furthermore, the $\twoq$ vectors $\rj$ corresponding to these
$\al_j$ have to span the subspace $\R^q$ of the $x$-variables.
Solving the \cglpq\ in this subspace yields $\twoq$ values $\al_j$
and $q \times \twoq$ associated multipliers $u_{ik}$, $i =
1,\ldots,\twoq$, $k=1,\ldots,q$; and these multipliers can then be
used to compute the remaining components of $\al$, given by the expressions (\ref{eq3.11}) and (\ref{eq3.12}). The significance
of this is that the computational cost of generating facets of
$\convpd$ grows only linearly with $n$. Of course this cost grows
exponentially with $q$, but the approach discussed here is being
considered for small $q$.

The choice of the subspace is intimately related to the
question of deciding which facets of the
disjunctive hull are also facets of the integer hull. The best way
to address this question and that of the subspace to be chosen, is
to first interpret the inequalities defining the disjunctive hull as
intersection cuts.

\section{Geometric interpretation: Cuts from the $q$-dimensional parametric cross-polytope}
\label{chapter2_subsection_cuts_from_parametric_crosspolytope}

Consider the $q$-dimensional unit cube centered at
$\left(0,\ldots,0\right)$, $K_q := \{ x \in \R^q
: -\frac{1}{2} \le x_j \le \frac{1}{2}, \ j \in Q\}$. Its polar,
$K^o_q := \{ x \in \R^q : xy \le 1, \ \forall x \in K\}$, is known
to be the $q$-dimensional octahedron or cross-polytope; which, when
scaled so as to circumscribe the unit cube, is the outer polar of
$K_q$:
$$K^*_q = \{ x \in \R^q : |x| \le {\textstyle\frac{1}{2}} q\},
$$
where $|x|=\sum (|x_j|:j=1,\ldots,q\}.$ Equivalently, $|x| \le
\frac{1}{2}q$ can be written as the system
\begin{equation}\label{eq5.1}
\begin{array}{rcrcrcl}
-x_1 &-& \cdots &-& x_q & \le & \frac{1}{2}q\\[6pt]
x_1 &-& \cdots &-& x_q &\le& \frac{1}{2}q\\[6pt]
&&&&&\vdots& \\[6pt]
x_1 &+& \cdots &+& x_q & \le & \frac{1}{2}q
\end{array}
\end{equation}
of $t=\twoq$ inequalities in $q$ variables.

Moving the center of the coordinate system to $(\half,\cdots,\half)$ changes the righthand side coefficient of the
$i$-th inequality in (\ref{eq5.1}) from $\half q$ to a value equal to the sum of positive coefficients
on the lefthand side of the inequality. Indeed, if $\qplus$ and $\qmin$ denotes the number of positive and negative coefficients, then $\frac{1}{2} q + \frac{1}{2} \qplus - \frac{1}{2} \qmin = \qplus$.

Next we introduce the parameters $v_{ik}$, $i=1,\ldots,t=\twoq$,
$k=1,\ldots,q$, to obtain the system
\begin{equation}\label{eq5.2}
\begin{array}{rcccrcl}
-v_{11} x_1 &-& \cdots &-& v_{1q} x_q & \le & 0\\[6pt]
v_{21} x_1 &-& \cdots &-& v_{2q} x_q &\le& v_{21}\\[6pt]
-v_{31} x_1 &+& \cdots &-& v_{3q} x_q &\le& v_{31}\\[6pt]
\vdots~~~ &&&&&\vdots&\\[6pt]
v_{t1} x_1 &+& \cdots &+& v_{tq} x_q &\le& v_{t1} + \ldots + v_{tq}\\[6pt]
\multicolumn{7}{l}{v_{ik} \ge 0, \ i=1,\ldots,t=\twoq, \
k=1,\ldots,q.}
\end{array}
\end{equation}
Note that the constraints of (\ref{eq5.2}) are of the form
$$\sum_{k \in \widetilde{Q}^+_i} v_{ik} \xk - \sum_{k \in \widetilde{Q}^-_i} v_{ik} \xk \le
\sum_{k \in \widetilde{Q}^+_i} v_{ik},
$$
where $\widetilde{Q}^+_i$ and $\widetilde{Q}^-_i$ are the sets of
indices for which the coefficient of $x_k$ is $+ v_{ik}$ and $-
v_{ik}$, respectively. Note also that all inequalities that have the
same number of coefficients with the plus sign have the same
righthand side, equal to the sum of these coefficients.

The system (\ref{eq5.2}) is homogeneous in the parameters $v_{ik}$, so every one of its inequalities can be normalized. Since we are looking for a
connection with the system (\ref{eq3.8}) defining \cglpq, we will
use the normalization given by this system and
Proposition~\ref{pr4}, i.e.
\begin{equation}\label{eq5.3}
\begin{array}{rcccrcl}
f_1v_{11} & + & \cdots & + & f_q v_{1q} & = & 1\\[6pt]
(1-f_1)v_{21} &+& \cdots &+& f_q v_{2q} &=& 1\\[6pt]
\cdots &&&& \cdots\\[6pt]
(1-f_1) v_{t1} &+& \cdots &+& (1-f_q) v_{tq} &=& 1
\end{array}
\end{equation}
Note that these normalization constraints are of the general form
$$\sum_{h \in \widetilde{Q}^+_i} (1-f_k) v_{ik} + \sum_{h \in
\widetilde{Q}^-_i} f_k v_{ik} =1.
$$
Let $\tkv$ denote the parametric cross-polytope defined by
(\ref{eq5.2}) and (\ref{eq5.3}). It is not hard to see that for any fixed set of
$v_{ik}$, (\ref{eq5.2}) defines a convex polyhedron in $x$-space
that contains in its boundary all $x \in \R^q$ such that $\xk \in
\{0,1\}$, $k \in Q$, hence is suitable for generating intersection
cuts. Furthermore, letting $\tkvn$ be the expression for $\tkv$ in
the space of the $s$-variables, obtained by substituting $f+Rs$ for
$x$ into (\ref{eq5.2}), we have

\begin{theorem}\label{th6.1}
For any values of the parameters $v_{ik}$ satisfying (\ref{eq5.2}) and (\ref{eq5.3}),
the intersection cut $\tal s \ge 1$ from $\tkvn$ has coefficients
$\tal_j= \frac{1}{\sj^*}$, where
\begin{equation}\label{eq5.4}
\sj^* = \max \{ \sj : f + \rj \sj \in \kvn\}.
\end{equation}
\end{theorem}

\begin{proof}
This is a special case of Theorem~\ref{th1.1}.
\end{proof}

In order to compare the intersection cut $\tal s \ge 1$ with the cut
$\al s \ge 1$ from the $q$-term disjunction (\ref{eq3.7}), we have
to restate (\ref{eq5.4}) in terms of the system of inequalities
defining $\tkvn$. This means that $f+\rj \sj^*$ has to be expressed
as the intersection point of the ray $f+\rj \sj$, $\sj \ge 0$, with
the first facet of $\kvn$ encountered. This yields
\begin{equation}\label{eq5.5}
\sj^* = \min \{ \sj^1,\ldots,\sj^t\},
\end{equation}
where the $\sj^i$ are obtained by substituting $f_k + \sum^n_{h=1}
r_j^k s_h$ for $x_k$, $k=1,\ldots,q$ into the $i$-th inequality of
(\ref{eq5.2}), and setting $s_h=0$ for all $h \ne j$:
$$s^i_j = \max \left\{ \sj : \left(\sum_{k \in \tqpi} v_{ik} r^k_j - \sum_{k
\in \tqmi} v_{ik} r^k_j\right) \sj \le \sum_{k \in \tqpi} v_{ik}
(1-f_k) + \sum_{k \in \tqmi} v_{ik} f_k\right\},
$$
$i=1,\ldots,t=\twoq$.

Clearly, this maximum is bounded whenever the coefficient of $\sj$
is positive, in which case, if we normalize by setting $\sum_{k \in
\tqpi} v_{ik} (1-f_k) + \sum_{k \in \tqmi} v_{ik} f_k =1$, we obtain
\begin{equation}\label{eq5.6}
\sj^i = \left( \sum_{k \in \tqpi} v_{ik} r_j^k - \sum_{k \in \tqmi}
v_{ik} r_j^k \right)^{-1}.
\end{equation}

Comparing (\ref{eq5.5}) and (\ref{eq5.6}) to the expressions
(\ref{eq3.11}) and (\ref{eq3.12}) for the coefficient $\al_j$ of the
lift-and-project cut $\al s \ge 1$ of Proposition~\ref{pr2},
we find that setting $v_{ik} = u_{ik}$ for all $i,k$, as well as $\tqpi = \Qmi$
and $\tqmi = \Qpi$, we obtain $\tal_j = \al_j$.

This proves

\begin{corollary}\label{co6.2}
The intersection cut $\tal s \ge 1$ from the parametric octahedron
$\tkvn$ is the same as the lift-and-project cut $\al s \ge 1$
corresponding to the \cglpq\ solution $(\al,u)$, with
$v_{ik}=u_{ik}$, $i=1,\ldots,t$, $k=1,\ldots,q$.
\end{corollary}

\section{Facets of the disjunctive hull and the integer hull}
\label{chapter2_subsection_facets_dh_ih}

Consider again the disjunctive relaxation of $\PI$
$$\PD = \{ (x,s) \in \Rq \times \Rn : x = f + Rs, \ s \ge 0, \ x_i
\le 0 \ \vee \ x_i \ge 1, \ i \in Q \}
$$
introduced at the beginning of section \ref{sec:integer_and_disjunctive_hulls}, where $x,f \in \Rq$, $R \in \R^{q \times
n}$, and $Q := \{ 1,\ldots,q\}$. For $i=1,\ldots,t=\twoq$, let $p^i$
be the vertex of $K_q$, the $q$-dimensional unit cube, defined by
$p^i_k =0$, $i \in \Qpi$, $p^i_k = 1$, $i \in \Qmi$.

Next we give a sufficient condition for an inequality $\al s \ge 1$
valid for $\PD$ to define a facet of $\convpi$, which for small $q$ leads to an efficient procedure for
generating inequalities that are facet defining for $\convpi$.

The dimension of $\PnI$ being $n \ge \twoq$, $\al s \ge 1$ defines a
facet of $\convpnI$ if there exists a subspace $\R^{\twoq}$ of $\Rn$
such that the restriction of $\al s \ge 1$ to this subspace defines
a facet of $\convpi^{(\twoq)}$. If this is the case, then the
inequality in question can be lifted to the full space to yield a
facet of $\convpnI$ by using the $u$-components of the solution
$(\al,u)$ to the CGLP in the subspace to compute the missing
coefficients $\al_j$.

\begin{theorem}\label{th7.1} 
Let $\al s \ge 1$ be a valid inequality for $\PD$ corresponding to a
basic solution $(\al,u)$ of \cglpq, and let $p^i$,
$i=1,\dots,\twoq$, be the vertices of $K_q$. Suppose for each $p^i$, $i=1,\ldots,2^q$, there exists a subset $J_i \subset J$ containing the indices of $q$ linearly independent rays $r^{j_1},\ldots,r^{j_q}$, and a vector $\lambda \in \Rq_+$, satisfying
\begin{equation}\label{eq6.1}
p^i - f = \sum^{j_q}_{j=j_1} {\textstyle\frac{1}{\alpha_j}} \rj
\lambda_j, \quad \sum^{j_q}_{j=j_1} \lambda_j = 1.
\end{equation}
Then the inequality $\sum_{\jinJ} \al_j \sj \ge 1$ defines a facet
of $\convpi^{(|J|)}$, and its lifting based on the $u$-components of
the solution $(\al,u)$ defines a facet of $\convpnI$.
\end{theorem}

\begin{proof}
Suppose the subset of $\twoq$ rays indexed by $J$ satisfies the
requirements of the Theorem. Then for every $i=1,\ldots,\twoq$, the
vertex $p^i$ of $K^q$ satisfies
$$p^i = \sum^{j_q}_{j=j_1} (f - {\textstyle\frac{1}{\al_j}} \rj)
\lambda_j, \quad \sum^{j_q}_{j=j_1} \lambda_j =1
$$
for some $\lambda_j \ge 0$, $j=j_1,\ldots,j_q$, i.e. $p^i$ can be
expressed as a convex combination of the $q$ points
$f+\frac{1}{\al_j}\rj$, $j=j_1,\ldots,j_q$. But $f + \frac{1}{\al_j}
\rj = f+ \rj \sj^*$ is the intersection point of the ray $f+\rj\sj$
with bd\,$\widetilde{K}^*_q$, hence each of these points satisfies
$\al s=1$ and consequently so does $p^i$. Since $\sum\limits_{\jinJ}
\al_j \sj \ge 1$ is satisfied at equality by $\twoq$ integer points
of $\convpi^{(|J|)}$, it defines a facet of the latter. Furthermore,
lifting the remaining coefficients $\al_j$ of the inequality by
using the $u$-components of $(\al,u)$ yields a facet defining
inequality for $\convpnI$.
\end{proof}

The sufficient condition of Theorem~\ref{th7.1} is not necessary.
There are two kinds of situations not satisfying the above
condition, in which a valid inequality $\al s \ge 1$ for $\PD$ may
define a facet of $\convpi$. The first one involves an inequality
$\al s \ge 1$ such that although (\ref{eq6.1}) is not satisfied for all
$\twoq$ vertices $p^i$ of $K^q$, nevertheless $\convpd$ has $\twoq$ vertices whose
$x$-components $p^i$ satisfy (\ref{eq6.1}), i.e. $\convpd$ has
multiple vertices with the same $x$-component. The second situation
involves facet defining split cuts.

\section{The two-row case}
\label{chapter2_section_case_2_dimensions}

We now restrict our attention to the case $q=2$, i.e. we consider two rows from a simplex tableau of a MIP problem with the variables $x_1, x_2$ and $s_j,j\in J$:
\begin{equation}
\label{2rowstableau}
\begin{array}{llllll}
P_L = \{ (x,s) \in \mathbb{R}^{2+|J|} : \ & x_1 & = f_1 + \sum_{j\in J} r_j^1 s_j\\[6pt]
&x_2 & = f_2 + \sum_{j\in J} r_j^2 s_j\\[6pt]
&\multicolumn{3}{l}{s_j  \geq 0 \quad j\in J\ \}.}
\end{array}
\end{equation}
where $x_1, x_2$ are basic variables required to be integers and $s_j,j\in J$ are non-basic. This is the case studied by Anderson, Louveaux, Weismantel and Wolsey \cite{ALWW}. Let $P_I = \{ (x,s) \in \mathbb{Z}^{2} \times \mathbb{R}^{|J|} : (x,s) \in P_L\}$, and $0<f_1, f_2 < 1$. The column vectors $r_j,j\in J$, represent the extreme rays of the cone in $\R^{|J|}$ with apex at $(f_1,f_2)$.

We will say that a ray $r_j$ in (\ref{2rowstableau}) hits an orthant-cone $Q_i,i\in\{1,\dots,4\}$
if there exists $\lambda_0 > 0$ such that $f + \lambda r_j \in Q_i$ for all $\lambda \geq \lambda_0$.

For the case of 2 rows the disjunction (\ref{eq3.7}) becomes
\begin{equation}
\label{nonbasicdisjunction}
\begin{pmatrix}
-r^1 s \geq f_1\\
-r^2 s \geq f_2\\
\end{pmatrix}
\vee
\begin{pmatrix}
r^1 s \geq 1 - f_1\\
-r^2 s \geq f_2\\
\end{pmatrix}
\vee
\begin{pmatrix}
r^1 s \geq 1 - f_1\\
r^2 s \geq 1 - f_2\\
\end{pmatrix}\\
\vee
\begin{pmatrix}
-r^1 s \geq f_1\\
r^2 s \geq 1 - f_2\\
\end{pmatrix}
\end{equation}
with $s\geq 0$, and the system (\ref{eq3.8}) of Theorem~\ref{th5.1} becomes
\begin{equation}
\label{cglpmip_normalized}
\begin{array}{lllllll}
\alpha & +r^1 v_1 & + r^2 w_1 &             \geq 0\\
\alpha & -r^1 v_2 & + r^2 w_2 &             \geq 0\\
\alpha & -r^1 v_3 &  - r^2 w_3 &            \geq 0\\
\alpha & +r^1 v_4 &  - r^2 w_4 &            \geq 0\\
 & +f_1 v_1 &  +f_2 w_1 &            = 1\\
 & +(1-f_1) v_2 &  +f_2 w_2 &        = 1\\
 & +(1-f_1) v_3 &  +(1-f_2) w_3 &    = 1\\
 & +f_1 v_4  & +(1-f_2) w_4 &        = 1\\
\multicolumn{3}{l}{v_i,w_i \geq 0\quad i\in \{1\dots 4\}.}
\end{array}
\end{equation}
where $v_i, w_i$, $i = 1,\ldots,4$ stand for $u_{i1}, u_{i2}$, $i = 1,\ldots,t=2^q$ (since $q=2$, $t = 2^q = 4$).

By Proposition \ref{pr2} the cuts generated by the CGLP with constraint set (\ref{cglpmip_normalized}) and objective function $\min p\al$ for some $p > 0$ have the form $\alpha s \geq 1$, where
\begin{equation*}
\alpha_j = \max \{ \alpha_j^1,\alpha_j^2,\alpha_j^3,\alpha_j^4 \}
\end{equation*}
with
\begin{equation}
\label{alphas_expression}
\begin{array}{lll}
\alpha_j^1 & = -r^1_j v_1 & - r^2_j w_1\\
\alpha_j^2 & = +r^1_j v_2 & - r^2_j w_2\\
\alpha_j^3 & = +r^1_j v_3 &  + r^2_j w_3\\
\alpha_j^4 & = -r^1_j v_4 &  + r^2_j w_4.
\end{array}
\end{equation}

As discussed in Section~\ref{chapter2_subsection_cuts_from_parametric_crosspolytope}, a cut produced by the CGLP can be viewed as an intersection cut derived from a parametric cross-polytope or octahedron. For given $v,w$, we call the polyhedron
\begin{equation*}
\begin{array}{lll}
P_{\text{octa}}(v,w) = \{ (x_1,x_2) \in \mathbb{R}^2 : & - v_1 x_1 - w_1 x_2 \leq 0 \ ; & \\
    & + v_2 x_1 - w_2 x_2 \leq v_2 \ ;        & \\
    & + v_3 x_1 + w_3 x_2 \leq v_3 + w_3 \ ; & \\
    & - v_4 x_1 + w_4 x_2 \leq w_4 	  & \}\\
\end{array}
\end{equation*}
the {\em $(v,w)$-parametric octahedron}.

If $v_i=0$ or $w_i=0$ for some $i\in \{1,\dots, 4\}$ the $i$-th facet of $P_{\text{octa}}$ is parallel to one of the coordinate axes. If $v_i,w_i >0$ then the $i$-th facet of $P_{\text{octa}}$ is \textit{tilted} (note that since we use the normalization $\beta = 1$, $v_i$ and $w_i$ cannot both be 0). Varying the parameters $v,w$, the $(v,w)$-parametric octahedron produces different configurations according to the non-zero components of $v,w$. Depending on the values taken by the parameters, $P_{\text{octa}}(v,w)$ may be a quadrilateral (i.e. a full-fledged octahedron in $\R^2$), a triangle, or an infinite strip. In the rest of the section we refer to these configurations using the short reference indicated in parenthesis. It can easily be verified that the value-configurations of the parameters $v_i,w_i$ which give rise to maximal convex sets are the following:

\begin{itemize}
\item ($S$) If exactly 4 components of $(v,w)$ are positive, \pocta\ is the vertical strip $\{x\in\mathbb{R}^2:0\leq x_1\leq 1\}$ if $v_i > 0, i=1,\dots,4$;
or the horizontal strip $\{x \in \R^2: 0 \le x_2 \le 1\}$ if $w_i > 0, i=1,\dots,4$ (see figure \ref{parocta_vstrip}, \ref{parocta_hstrip}).

\item ($T_A$) If exactly 5 components of $(v,w)$ are positive, \pocta\ is a triangle with 1 tilted face (type A) (by ``tilted'' we mean a face that is not parallel to any of the two axes). Figure \ref{parocta_5nz}  illustrates the case with $v_1,w_2,v_3,w_3,v_4>0;w_1,v_2,w_4=0$. When in addition $v_i=w_i$ for some $i\in\{1,\dots,4\}$ \pocta\ becomes a triangle with vertices $(0,0);(2,0);(0,2)$ or one of the other three configurations symmetric to this one. This corresponds to what is called a triangle of type 1 in \cite{DeyWolsey}. In the general case $T_A$ corresponds to a triangle of type 2 in \cite{DeyWolsey}.

\item ($T_B$) If exactly 6 components of $(v,w)$ are positive, \pocta\ is a triangle with 2 tilted faces (type B). Figure \ref{parocta_6nz}  illustrates the case with $v_1,w_1,v_2,w_2,w_3,w_4>0; v_3,v_4=0$. This configuration corresponds to a triangle of type 2 in \cite{DeyWolsey}.

\item ($Q$) If all 8 components of $(v,w)$ are positive, \pocta\ is a quadrilateral. See Figure \ref{parocta_8nz}.
\end{itemize}

The case with 7 components of $(v,w)$ positive does not correspond to a maximal parametric octahedron, therefore we do not need to consider it. Suppose all the components are positive except for $v_1$ which is 0. The facet of $P_{\text{octa}}$ corresponding to $(0,0)$ is horizontal and goes through the point $(1,0)$. Is not hard to see that setting $v_2=0$ we enlarge the set defined by the parametric octahedron.

\begin{figure}[htbp]
\centering
\begin{tabular}{ccc}
\subfigure[4 non-zeros - vertical strip]{
\includegraphics[width=50mm]{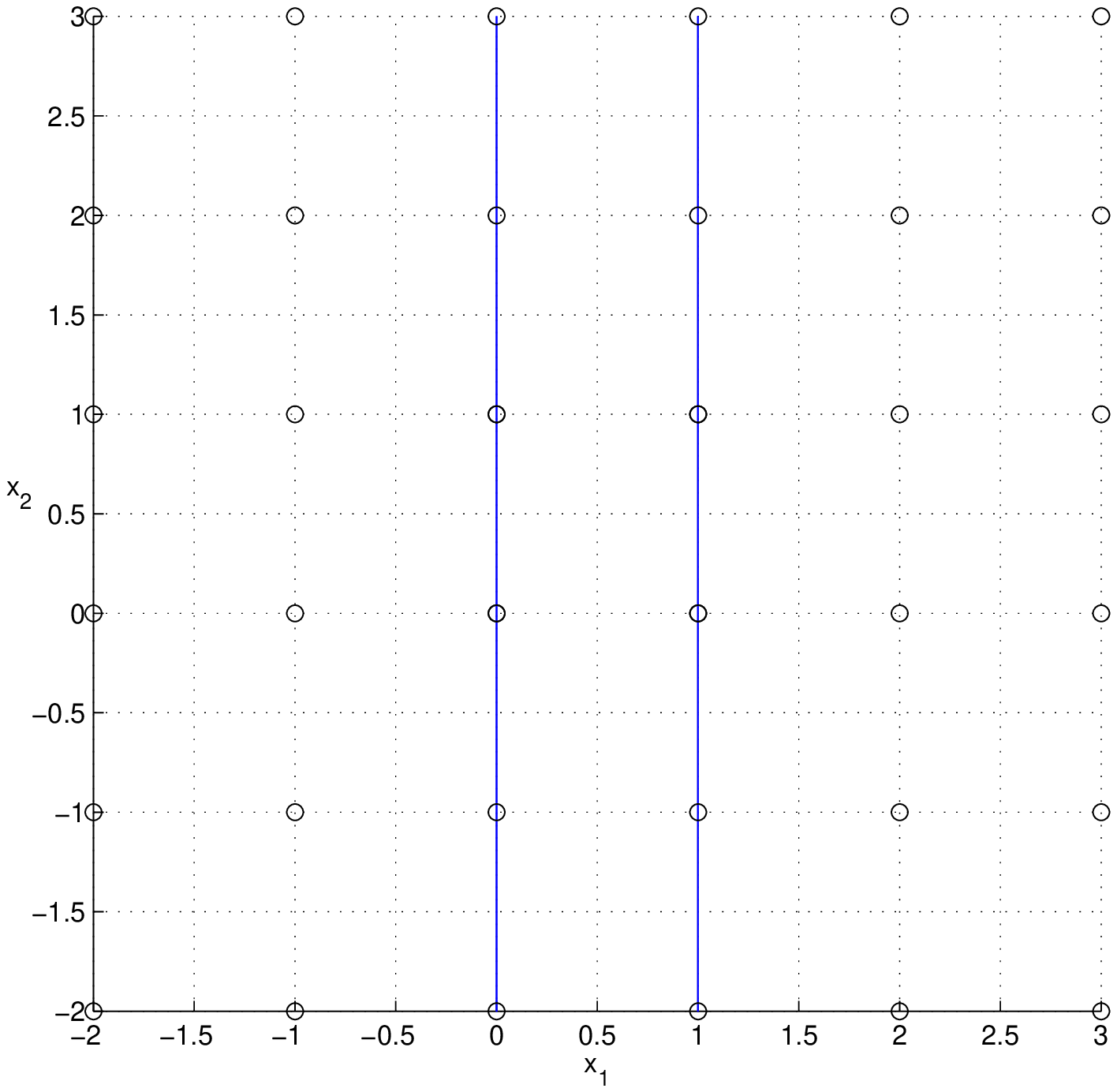}
\label{parocta_vstrip}
}
\subfigure[4 non-zeros - horizontal strip]{
\includegraphics[width=50mm]{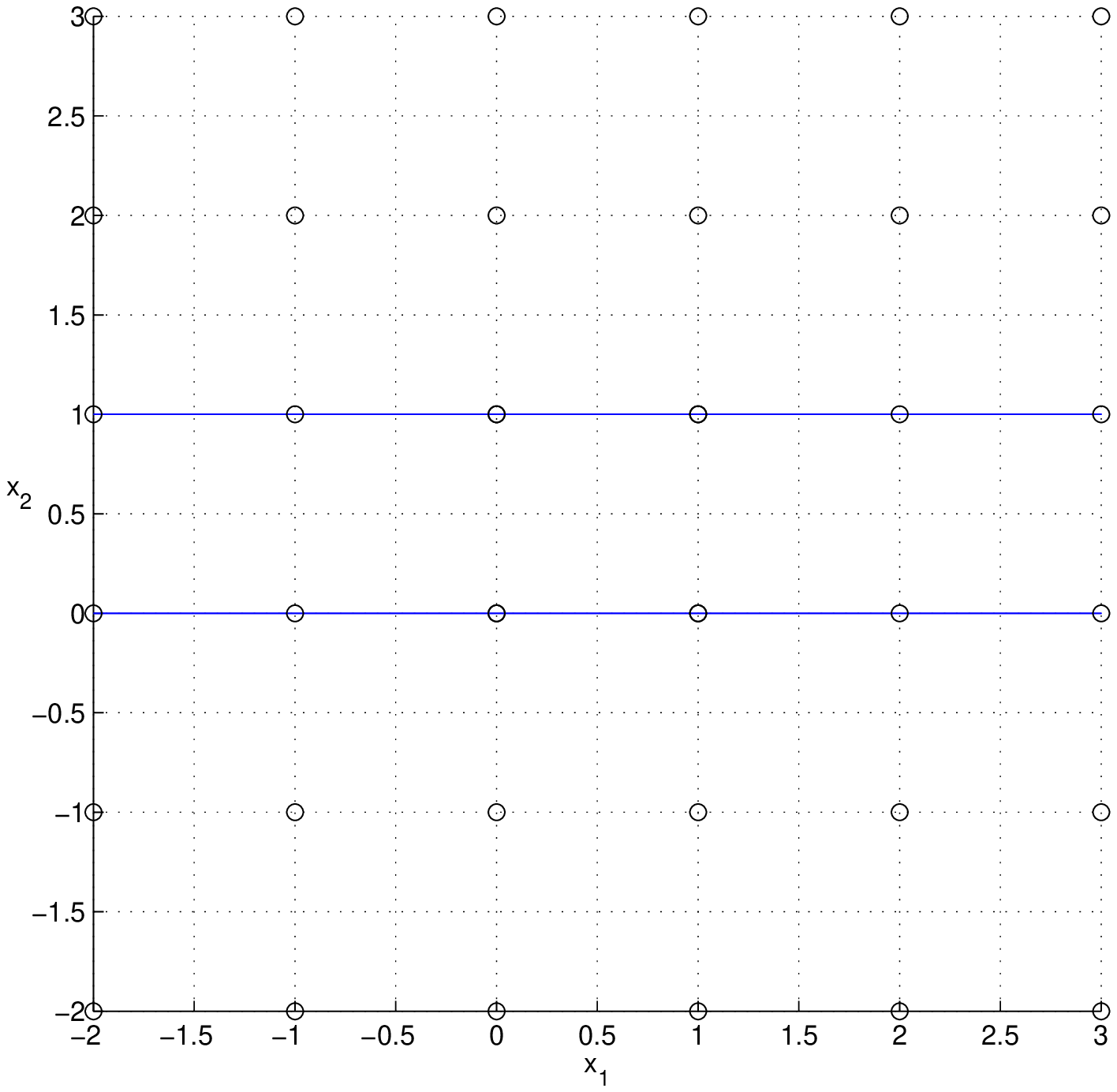}
\label{parocta_hstrip}
}
\subfigure[5 non-zeros - triangle of type A]{
\includegraphics[width=50mm]{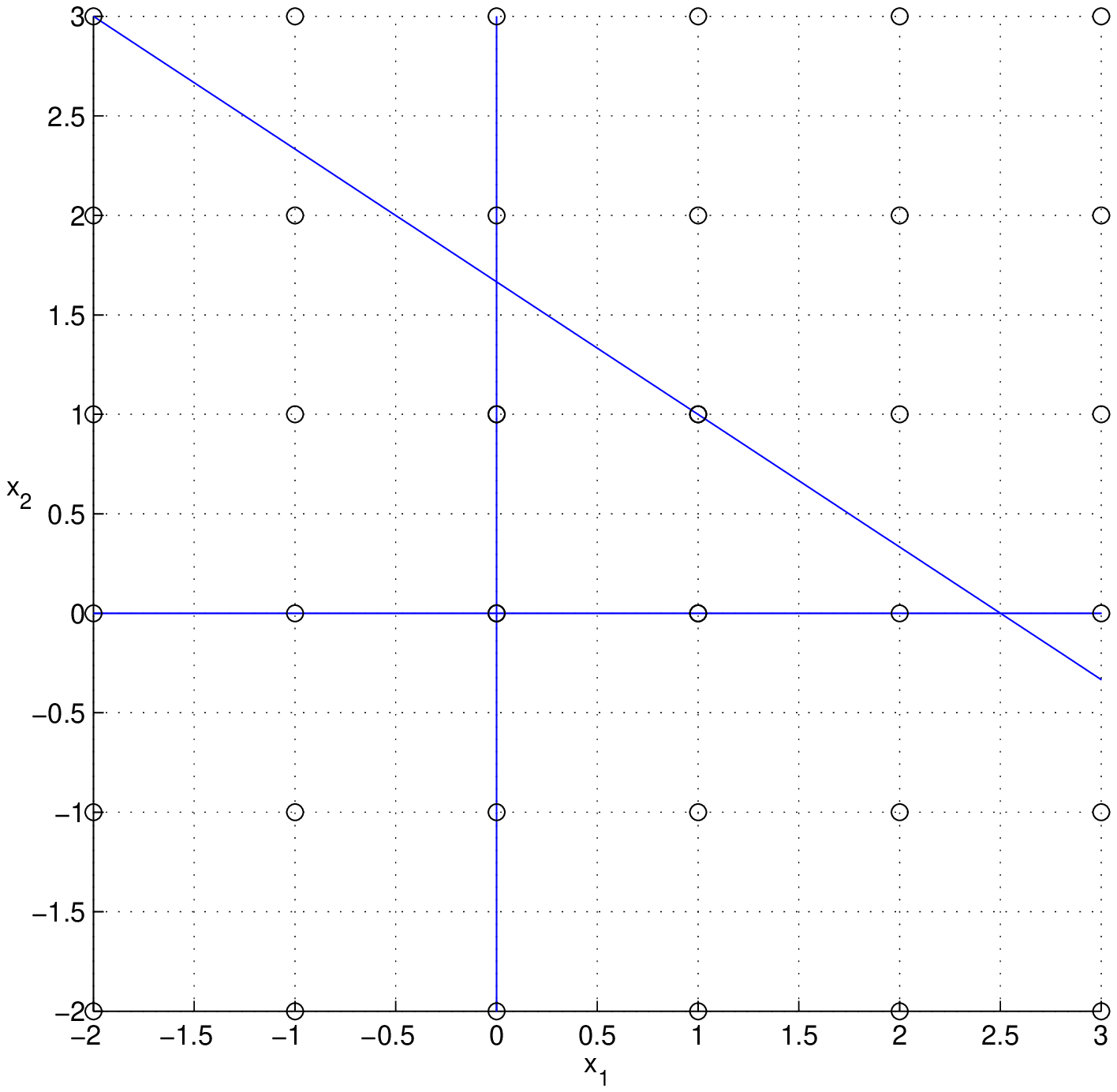}
\label{parocta_5nz}
}\\
\subfigure[6 non-zeros - triangle of type B]{
\includegraphics[width=50mm]{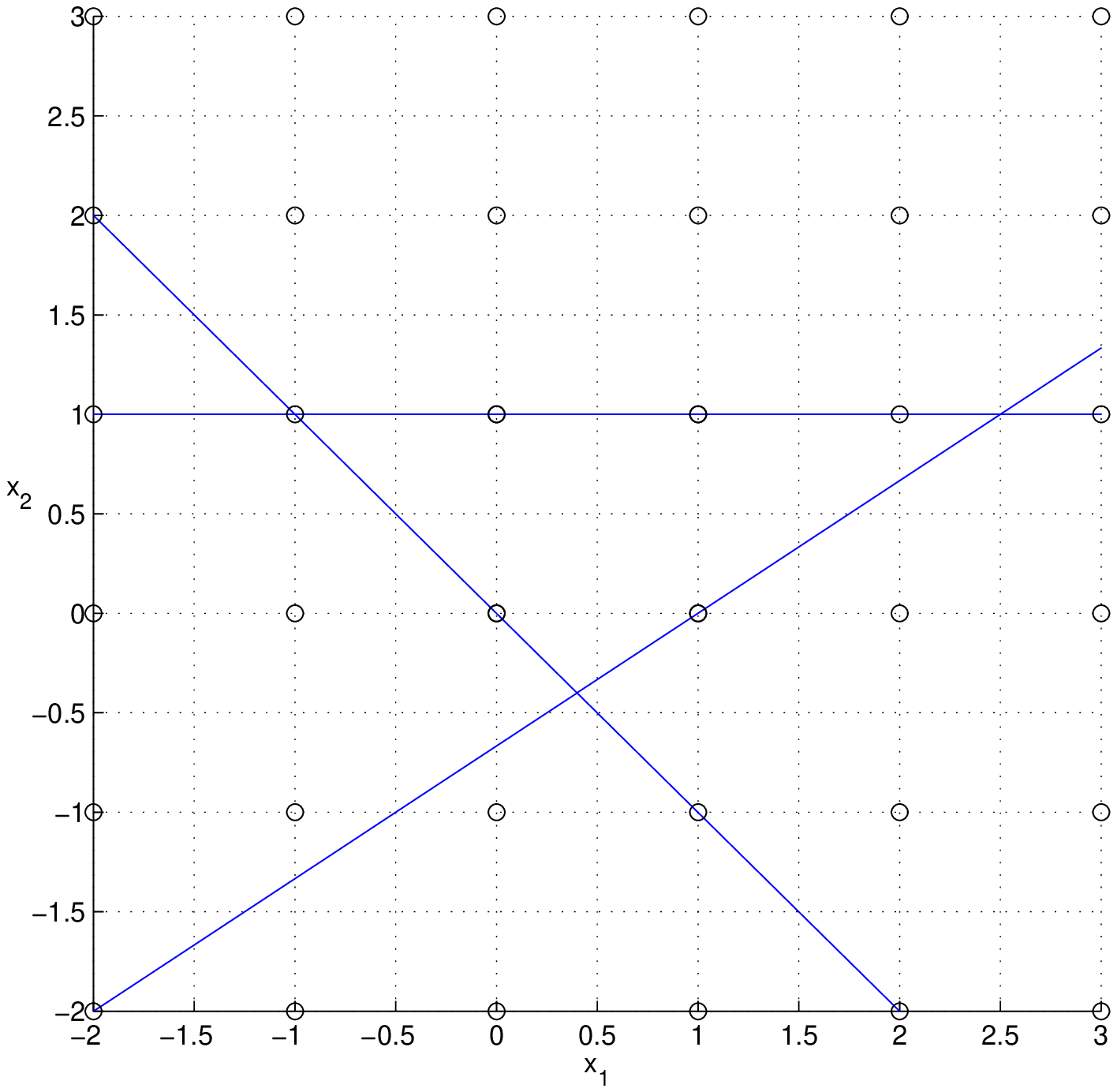}
\label{parocta_6nz}
}
\subfigure[8 non-zeros - quadrilateral]{
\includegraphics[width=50mm]{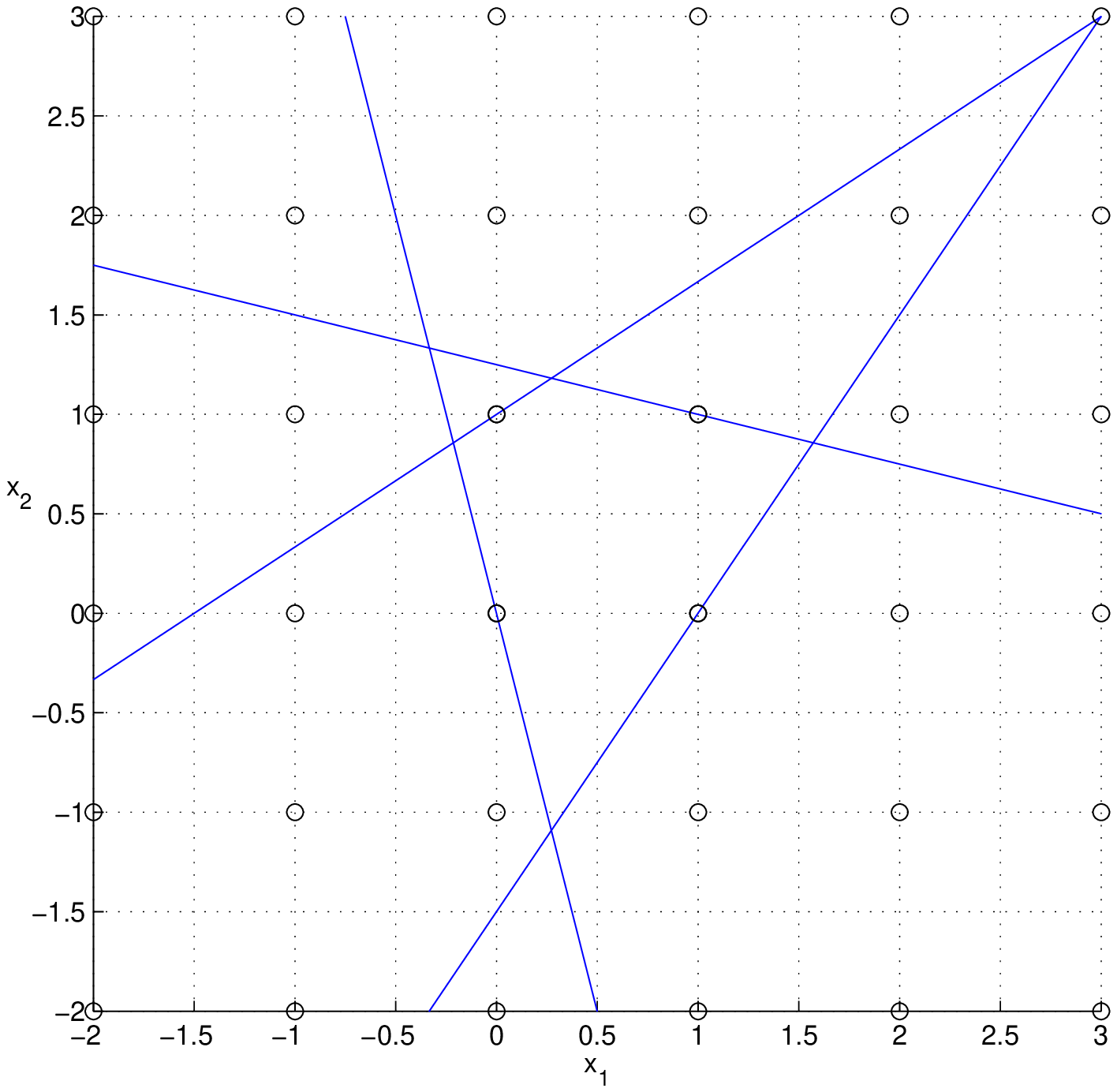}
\label{parocta_8nz}
}
\end{tabular}
\label{parocta_configurations}
\caption{Configurations of the parametric octahedron for the MIP case}
\end{figure}

For a cut $\sum_{j\in J} \alpha_j s_j \geq 1$  Andersen et al. \cite{ALWW} introduce the set
\begin{equation}
\label{Lalphadefinition}
 L_\alpha=\left\{ x \in \mathbb{R}^2  : (x,s) \in P_L \wedge \sum_{j\in J} \alpha_j s_j \leq 1 \right\}.
\end{equation}
Clearly, $L_\alpha \subseteq P_{\text{octa}}(v,w)$, and the inclusion is often strict.

\noindent{\bf Example.}

In \cite{ALWW}, Andersen et al. considered the two rows instance
\begin{equation}
\begin{array}{lllllllll}
x_1 & = \frac{1}{4} & + 2  s_1 & + 1  s_2 & - 3  s_3 &		& + 1  s_5\\[6pt]
x_2 & = \frac{1}{2} & + 1 s_1  & + 1  s_2 & + 2  s_3 & - 1  s_4 & - 2  s_5,\\[6pt]
\multicolumn{9}{c}{x_1, x_2 \in \Z, \quad s \ge 0}
\end{array}
\label{ALWW_example}
\end{equation}
We present the complete description of the disjunctive hull for (\ref{ALWW_example}). In order to do so we generated the CGLP of (\ref{ALWW_example}) using the normalization constraint $\beta = 1$ and we considered all feasible bases. The CGLP produces 5 different facets. For each of these we show the configuration of the parametric octahedron that yields the corresponding cut in terms of the $v,w$ variables:
\begin{enumerate}
 \item Cut ($T_B$): $2 s_1 + 2 s_2 + 4 s_3 + s_4 + \frac{12}{7} s_5 \geq 1$\\
            $v_1 = 2 ;\ v_2 = \frac{8}{7} ;\ v_3 = 0;\ v_4 = 0$\\
            $w_1 = 1;\ w_2 = \frac{2}{7} ;\ w_3 = 2;\ w_4 = 2$
 \item Cut ($T_B$): $\frac{8}{3} s_1 + \frac{4}{3} s_2 + \frac{44}{9} s_3 + \frac{8}{9} s_4 + \frac{4}{3} s_5 \geq 1$\\
            $v_1 = \frac{20}{9} ;\ v_2 = \frac{4}{3}  ;\ v_3 = \frac{4}{3} ;\ v_4 = \frac{4}{9} $\\
            $w_1 = \frac{8}{9}  ;\ w_2 =  0;\ w_3 = 0  ;\ w_4 = \frac{16}{9} $
 \item Cut ($T_A$): $\frac{8}{3} s_1 + 2 s_2 + 4 s_3 + s_4 + \frac{4}{3} s_5 \geq 1$ \\
            $v_1 = 2 ;\ v_2 = \frac{4}{3}  ;\ v_3 = 0 ;\ v_4 = 0 $\\
            $w_1 = 1 ;\ w_2 = 0 ;\ w_3 = 2 ;\ w_4 = 2 $
 \item Cut ($S$): $\frac{8}{3} s_1 + \frac{4}{3} s_2 + 12 s_3  + \frac{4}{3} s_5 \geq 1$ \\
            $v_1 = 4 ;\ v_2 = \frac{4}{3}  ;\ v_3 = \frac{4}{3}  ;\ v_4 = 4  $\\
            $w_1 = 0 ;\ w_2 = 0 ;\ w_3 = 0  ;\ w_4 = 0 $
 \item Cut ($T_B$): $2 s_1 + 2 s_2 + \frac{68}{7} s_3 + \frac{2}{7} s_4 + \frac{12}{7} s_5 \geq 1$\\
            $v_1 = \frac{24}{7}  ;\ v_2 = \frac{8}{7} ;\ v_3 = 0 ;\ v_4 = 0$\\
            $w_1 = \frac{2}{7} ;\ w_2 = \frac{2}{7} ;\ w_3 = 2 ;\ w_4 = 2$
\end{enumerate}

Of the 5 facets of $P_{D}$, 3 are facets for $P_{I}$: cuts 1, 2 and 4. Note that cut 4 is a split cut and can be derived using only the tableau row corresponding to the variable $x_2$. Cut 3 and 5 are facets of $P_D$ by Theorem~\ref{th1}.

The condition given in Theorem~\ref{th7.1} for an inequality $\al x \ge 1$, facet defining for the disjunctive hull, to also define a facet of the integer hull specializes for the case $q=2$ to the following. For each of the four vertices $p^i$ of $K$, $p^i$ must lie on the line segment between two intersection points of rays $r^j$ with the boundary of \pocta. As discussed in Section \ref{chapter2_subsection_cuts_from_parametric_crosspolytope}, the inequalities $\al x \ge 1$ can be generated in a subspace of $\le 2^q = 4$ variables, and then lifted into the full space by using the multipliers $(v_i,w_i)$, $i = 1,\ldots,4$. In \cite{Qualizza_dissertation} two algorithms were implemented for generating facets of the integer hull from \pocta, one for the case of a quadrilateral, the other for the case of triangles, both of them linear in $|J|$, the number of rays.

Recently Dash et al. \cite{DashDeyGunluk} have generalized the approach of \cite{BalasDH1, BalasDH2}, by considering more general 4-term disjunctions that give rise to what they call cross cuts and crooked cross cuts. They relate the closures of their cuts with the split closure and show, among others, that any 2 dimensional lattice free cut can be obtained as a crooked cross cut.

\section{Cut Strengthening}
\label{chapter2_subsection_strengthening}

Given a facet $\alpha s \geq 1$ of the disjunctive hull, if some non-basic variable $s_j$ is required to be integral in the original problem formulation, then the cut can be strengthened. Let $J_1$ be the index set of the integer-constrained variables $\sj$, and let $J_2 = J \sm J_1$.

\begin{lemma}\label{lem2.3.5}
If the disjunction
\begin{equation}
\label{nonbasicdisjunction_integralnonbasics}
\begin{pmatrix}
-r^1 s \geq f_1\\[6pt]
-r^2 s \geq f_2\\
\end{pmatrix}
\vee
\begin{pmatrix}
r^1 s \geq 1 - f_1\\[6pt]
-r^2 s \geq f_2~~~~~~~\\
\end{pmatrix}
\vee
\begin{pmatrix}
r^1 s \geq 1 - f_1\\[6pt]
r^2 s \geq 1 - f_2\\
\end{pmatrix}\\
\vee
\begin{pmatrix}
-r^1 s \geq f_1~~~~~~~\\[6pt]
r^2 s \geq 1 - f_2\\
\end{pmatrix}
\end{equation}
where $s \ge 0$ and $\sj \in \Z$, $j \in J_1 \subseteq J$, is valid for $\PI$, then so is the disjunction obtained from (\ref{nonbasicdisjunction_integralnonbasics}) by replacing some or all $r^i_j$, $i = 1, 2$, $j \in J_1$, with $r^i_j - m^i_j$, for any $m^i_j \in \Z$, $i = 1,2$, $j \in J_1$.
\end{lemma}

\begin{proof}
Suppose there exists $\istar \in \{ 1,2 \}$ and $\jstar \in J_1$ such that replacing $r^{\istar}_{\jstar}$ with $r^{\istar}_{\jstar} - \bar{m}^{\istar}_{\jstar}$, where $\bm^{\istar}_{\jstar} \in \Z$, violates (\ref{nonbasicdisjunction_integralnonbasics}). Then there exists a solution $(x,s) \in \PI$ with $x \in \Z^2$ such that
$$(- (\rijstar - \bmijstar)\sjs - \sum_{\jinJ\sm \{\jstar\}} r^{\istar}_j \sj < f_{\istar}) \ \ \wedge \ \ ((\rijstar - \bmijstar)\sjs + \sum_{\jinJ\sm\{\jstar\}} r^{\istar}_j \sj < 1 - f_{\istar})
$$
holds. Rewriting this expression so as to bring together the terms in $\bmijstar$ we get
$$\sum_{\jinJ} r^{\istar}_j \sj + f_{\istar} - 1 < \bmijstar \sjs < \sum_{\jinJ} r^{\istar}_j \sj + f_{\istar}
$$
or
$$-1 < \bmijstar < 0
$$
contrary to the fact that both $\bmijstar$ and $\sjs$ are integer.
\end{proof}

\begin{theorem}\label{th7.2}
Given $(\bar{v},\bar{w}) \geq 0$ defining a parametric octahedron, the cut $\alpha s \geq 1$ can be strengthened to $\bar{\alpha} s \geq 1$ with coefficients $\bar{\alpha}_j, j\in J_1$ given by the 3-variable mixed integer program
\begin{equation}
\label{strengthenedalphaj}
\begin{array}{lllllll}
\min & \alpha_j\\
& \alpha_j & - \bar{v}_1 m^1_j - \bar{w}_1 m^2_j& \geq & -r_j^1 \bar{v}_1 - r_j^2 \bar{w}_1\\
& \alpha_j & + \bar{v}_2 m^1_j - \bar{w}_2 m^2_j& \geq & +r_j^1 \bar{v}_2 - r_j^2 \bar{w}_2\\
& \alpha_j & + \bar{v}_3 m^1_j + \bar{w}_3 m^2_j& \geq & +r_j^1 \bar{v}_3 + r_j^2 \bar{w}_3\\
& \alpha_j & - \bar{v}_4 m^1_j + \bar{w}_4 m^2_j& \geq & -r_j^1 \bar{v}_4 + r_j^2 \bar{w}_4\\
& \multicolumn{3}{l}{m^1_j, m^2_j \in \mathbb{Z}.}
\end{array}
\end{equation}
The coefficients for $j \in J_2$ remain unchanged at $\bal_j = \al_j$ as in Proposition~\ref{pr2}.
\end{theorem}
\begin{proof}
Validity of $\bar{\alpha} s \geq 1$ follows from Lemma~\ref{lem2.3.5}.
\end{proof}

\begin{theorem}\label{th7.3}
The mixed integer program (\ref{strengthenedalphaj}) has an optimal solution ($\bal_j, \bm^1_j,\bm^2_j$) satisfying $\bm^i_j \in \{ \lfloor \br^i_j \rfloor, \lceil \br^i_j \rceil\}$, $i=1,2$.
\end{theorem}
\begin{proof}
Let $(\tal_j, \tm^i_j, \tm^2_j)$ be an optimal solution to the problem obtained from (\ref{strengthenedalphaj}) by adding the constraint $m^i_j \in \{ \lfloor \br^i_j \rfloor, \lceil \br^i_j \rceil\}$. We will show that this solution cannot be improved by replacing $\tm^1_j, \tm^2_j$ with any other pair of integers.

Consider the linear programming relaxation of (\ref{strengthenedalphaj}), which asks for minimizing the maximum of four linear functions. This is a piece-wise linear convex programming problem whose minimum is attained for $m^i_j = r^i_j$, $i = 1,2$, yielding $\al_j = \al^1_j = \ldots \ldots = \al^4_j = 0$. From the convexity of the objective function $\al ( m^1_j,m^2_j)$ it follows that the integer optimum occurs at one of the points $(m^1_j, m^2_j) \in \{ ( \lfloor r^1_j \rfloor, \lfloor r^2_j\rfloor)$, $(\lfloor r^1_j \rfloor, \lceil r^2_j \rceil)$, $( \lceil r^1_j \rceil, \lfloor r^2_j \rfloor)$, $(\lceil r^1_j\rceil, \lceil r^2_j \rceil)\}$. For suppose the optimum were to occur at some other point, say $(\hmoj,\hmtj)$, where $\hmoj = \ceilroj$ and $\hmtj = \ceilrtj + d_j$ for some $d_j > 0$. Then
$$\begin{array}{lcl}
~~~~~~~~~~~~~~~~~~~~~~~~~\al ( \ceilroj, \ceilrtj + d_j) & < & \al ( \ceilroj, \ceilrtj),\\[12pt]
~~~~~~~~~~~~~~~~~~~~~~~~~\al ( \ceilroj, \rtj) & < & \al (\ceilroj,\ceilrtj),\\
\multicolumn{3}{l}{\rm hence~~~~~~~~~~~~~~~~~~~~~~~~~~~~~~~~~~~~~~~~~~~~~~~~~~~~~~~~~~~~~~~~~~~~~~~~~~~~~~~~~~~~~~~~~~~~~~~~~~~~~~~~~~~~~~~~~~~~~~~~~~~~~~~}\\
~~~~~~~~~~~~~~~~~~~~~~~~~\al (\ceilroj, \ceilrtj) & > & \lambda \al ( \ceilroj, \rtj) + (1-\lambda) \al ( \ceilroj,\ceilrtj + d_j) \mbox{ for } 0 \le \lambda \le 1,
\end{array}
$$
i.e. the value of the minimum at a point which lies on the line between $(\ceilroj,\rtj)$ and $(\ceilroj,\ceilrtj + d_j)$ is larger than a convex combination of the values of the minimum at the endpoints of the line, contrary to the assumption that $\al ( \moj,\mtj)$ is a convex function.
\end{proof}

The operation of replacing $r^i_j$ by $r^i_j - m^i_j$ for some $m^i_j \in \Z$, $i = 1,2$, in the expression for $\al$, is called the modularization of $r^i_j$, or more generally, the modularization of the cut $\al x \ge 1$. Using $\mij \in \{ \lfloor \rij \rfloor, \lceil \rij \rceil \}$ is called the standard modularization. It can be shown (see below) that the mixed integer program (\ref{strengthenedalphaj}) attains its optimum for a standard modularization.

\begin{lemma}
\label{modularization_in_K}
There exists a standard modularization $\bar{r}$ of the ray $r$ such that
\begin{equation}
\label{point_in_K}
0\leq f_i + \bar{r}^i\leq 1, \quad \quad i\in\{1,2\}
\end{equation}
i.e. the point $(f+\bar{r})$ belongs to $K$.
\proof
If $f_i + r^i-\lfloor r^i \rfloor \leq 1$ then let $m^i=\lfloor r^i \rfloor$. Note that the condition $f_i + r^i-\lfloor r^i \rfloor \geq 0$ follows since $0\leq f_i\leq 1$ and $r^i - \lfloor r^i \rfloor \geq 0$.
Otherwise ($f_i + r^i-\lfloor r^i \rfloor > 1$) let $m^i=\lceil r^i \rceil$ and from $f_i \leq 1$ and $r^i-\lfloor r^i \rfloor \leq 1$ we get $0 \leq f_i + r^i - \lfloor r^i \rfloor -1 = f_i + r^i - \lceil r^i \rceil \leq 1$.
\end{lemma}

For $k=1,\ldots,4$, let $\bal^k_j$ be obtained from $\al^k_j$ of (\ref{alphas_expression}) by substituting $\br^{\,i}_j$ for $\rij$, $i = 1,2$. One can show that each $\bal^k_j$ is the convex combination of one of the expressions $\frac{-\broj}{f_1}$ or $\frac{\broj}{1-f_1}$ with one of the expressions $\frac{-\brtj}{f_2}$ or $\frac{\brtj}{1-f_2}$. To be specific, we have
\begin{lemma}\label{le7.5}
$$\begin{array}{rcll}
\bal ^1_j & = & \lambda_1 \frac{-\broj}{f_1} + (1 - \lambda_1 ) \frac{-\brtj}{f_2}  \quad , ~ & \mbox{with } \lambda_1 = \bv_1 f_1\\[12pt]
\bal^2_j & = & \lambda_2 \frac{\broj}{1-f_1} + (1 - \lambda_2 ) \frac{-\brtj}{f_2}  \quad , ~ & \mbox{with } \lambda_2 = \bv_2 (1-f_1)\\[12pt]
\bal^3_j & = & \lambda_3 \frac{\broj}{1-f_1} + (1 - \lambda_3 ) \frac{\brtj}{1-f_2} \quad , ~ & \mbox{with } \lambda_3 = \bv_3 (1-f_1)\\[12pt]
\bal^4_j & = & \lambda_4 \frac{-\broj}{f_1}  + (1 - \lambda_4 ) \frac{\brtj}{1-f_2} \quad , ~ & \mbox{with } \lambda_4 = \bv_4 f_1\\[12pt]
\end{array}
$$
\end{lemma}

\begin{proof}
By substituting for the $\lambda_k$, $k=1,\ldots,4$, we get the corresponding expressions for $\bal^k_j$.
\end{proof}

\begin{theorem}\label{th7.6}
The strengthened cut $\bal s \ge 1$ satisfies $0 \le \bal_j \le 1$, $j \in J_1$.
\end{theorem}

\begin{proof}
Since $\bv_k,\bw_k \ge 0$ for all $k$, we have $\bal^k_j \ge 0$ for at least one of the four $k$, hence $\bal_j \ge 0$. Let $(\bal, \bm^1, \bm^2)$ be an optimal solution to (\ref{strengthenedalphaj}). Let $\br^i = r^i - \bm^i$, $i = 1,2$, where $\bm^i \in \{ \floorri , \ceilri\}$ $i = 1,2$. There are four cases:

\noindent {\bf Case 1.} $\bm^i = \floorri, i = 1,2$. Then $\br^i = r^i - \bm^i \ge 0$, $i = 1,2$, and
\begin{tabbing}
$\bal^1$ \= $=$ \= $-\br^1 \bv_1 - \br^2 \bw_1 \le 0$.\\[6pt]
$\bal^2$ \> $=$ \> $ \br^1 \bv_2 - \br^2 \bw_2 \le \br^1/(1-f_1)$ (from (\ref{cglpmip_normalized})). From (\ref{point_in_K}), $\br^1 /(1 - f_1) \le 1$, hence $\bal^2 \le 1$.\\[6pt]
$\bal^3$ \> $=$ \> \parbox[t]{6in}{$\br^1 \bv_3 + \br^2 \bw_3  =  \lambda_3 \ \br^1/(1-f_1) + (1 - \lambda_3) \br^2/(1-f_2)$, with $\lambda_3 = \bv_3(1-f_1)$ (from Lemma~\ref{le7.5}). But from Lemma~\ref{modularization_in_K}, $\br^i/(1-f_i) \le 1$, $ i = 1,2$, hence $\bal^3 \le 1$.}\\[6pt]
$\bal^4$ \> $=$ \> $-\br^1\bv_4 + \br^2 \bw_4 \le \br^2/(1-f_2) \le 1$ (from \ref{point_in_K}), hence $\bal^4 \le 1$.
\end{tabbing}
The remaining three cases, namely $(\bm^1,\bm^2) = (\ceilbro,\floorbrt)$, $(\bm^1, \bm^2) = (\floorbro, \ceilbrt)$, and $\bm^i = ( \ceilbro, \ceilbrt)$, $i = 1,2$, are similar.
\end{proof}

A way to further strengthen these cuts consists in the following three-step procedure:
\begin{enumerate}
\item Apply standard modularization to each of the two rows from which the cut is generated (i.e. replace the ray $\rij$ by $\rij - \lfloor \rij \rfloor$ if $\rij > 0$ and by $\rij - \lceil \rij \rceil$ if $\rij < 0$, $i = 1, 2$, $j \in J_1$).
\item Generate a cut $\al x \ge 1$ from the two modularized rows.
\item Modularize the resulting cut to obtain the strengthened cut $\bal x \ge 1$.
\end{enumerate}

Yet another way to use the integrality of the variables $s_j$, $j \in J_1$, is to apply the monoidal cut strengthening procedure of \cite{BalasJeroslow}. For cuts generated from a disjunction of the form (\ref{nonbasicdisjunction_integralnonbasics}), this procedure involves the use of lower bounds on the expressions on the lefthand side of each inequality. While these bounds are readily available and quite tight in the case when $x_1, x_2 \in \{ 0,1 \}$, they can be weak in the general case of $x_1,x_2 \in \Z$. We therefore defer the dicussion of this procedure until the section on the 0-1 disjunctive hull.

\section{The 0-1 Disjunctive Hull}
\label{chapter2_section_01_disjunctive_hull}
We now consider the 0-1 disjunctive hull $\pde$ for $q=2$, i.e. we work with $P_{01} = \{ (x,s) \in \{0,1\}^{2} \times \mathbb{R}^{|J|} : (x,s) \in P_L\}$ where $P_L$ is given in (\ref{2rowstableau}). The CGLP that produces the facets of $\pde$ is the linear program with the constraint set of Theorem~\ref{th4.5}. In addition to the four configurations of the parametric octahedron for the MIP CGLP given in Section~\ref{chapter2_section_case_2_dimensions}, when $v,w$ are unrestricted in sign some additional configurations are possible: (a) triangles with each face containing exactly one vertex of $K$, which we call triangles of type C ($T_C$); and (b) cones, designated as $(C)$.

Note that our triangles of type C are similar to the class of triangles of type 3 for cuts for mixed integer programs described in \cite{DeyWolsey}. The difference between these classes is that on the one hand, the three integer points contained in the faces of triangles of type 3 defined in \cite{DeyWolsey} need not be vertices of $K$; on the other hand, our triangles of type C may also contain (non-0-1) integer points, positive or negative, in their interior. The presence among the parametric octahedra of unbounded ones, namely cones, implies that the cuts $\al s \ge 1$ of this class may have coefficients $\al_j < 0$.

As we did in Section~\ref{chapter2_section_case_2_dimensions}, we give a classification of the parametric cross-polytopes that correspond to disjunctive hull facets for the 0-1 case (i.e. facets of $\pde$). Let $k_1\in \{1,\dots,4\}$ be the index of any vertex of $K$. We denote by $k_2,k_3,k_4$ the indices of the vertices of $K$ that follow $k_1$ in counter-clockwise order. The following configurations, in addition to those for facets of $\PD$, are exhaustive when considering every value for $k_1\in\{1,\dots,4\}$ (mod 4) and swapping $v_i$ with $w_i$. In each case, the shape of \pocta\ is determined by a strict subset of the four pairs $(v_i,w_i)$, the remaining pairs being inactive.
\begin{itemize}
\item ($T_{C1}$) $v_{k_1}>0,w_{k_1}<0$; $v_{k_2},w_{k_2}>0$; $v_{k_3}>0,w_{k_3}>0$, $(v_{k_4}, w_{k_4}>0)$. \pocta\ is a triangle of type $C$ with all its vertices outside the cube $K$. The face corresponding to $k_4$ is inactive.  See Figure \ref{01parocta_t3t_cut2}.

\item ($T_{C2}$) $v_{k_1}>0,w_{k_1}<0$; $v_{k_2},w_{k_2}>0$; $v_{k_3}<0, w_{k_3}>0$,  $(v_{k_4}, w_{k_4}>0)$. \pocta\ is a triangle of type C with one vertex in the cube $K$. The face corresponding to $k_4$ is inactive. See Figure \ref{01parocta_t3t_cut}.

\item ($C_A$) $v_{k_2},v_{k_3}>0$; $w_{k_2}=w_{k_3}=0$, $v_{k_4}>0,w_{k_4}<0$, $(v_{k_1}, w_{k_1}>0)$. \pocta\ is a cone with one face containing two adjacent vertices of $K$, the other face containing one vertex of $K$. The face corresponding to $k_1$ is inactive. See Figure \ref{01parocta_cone_vertical}.

\item ($C_B$) $v_{k_1}<0, w_{k_1}>0$; $v_{k_3}>0, w_{k_3}<0$, $v_{k_4}, w_{k_4}>0$, $(v_{k_2},w_{k_2}>0)$. \pocta\ is a cone with one face containing two nonadjacent vertices of $K$, the other face containing one vertex of $K$. The face corresponding to $k_2$ is inactive. See Figure \ref{01parocta_cone_generic}.

\item ($C_C$) $v_{k_1}>0,w_{k_1}<0$; $v_{k_2},w_{k_2}>0$; $(v_{k_3}<0,w_{k_3}>0)$, $(v_{k_4}, w_{k_4}>0)$. \pocta\ is a cone with each face containing one vertex of $K$. The faces corresponding to $k_3$ and $k_4$ are inactive. See Figure \ref{01parocta_cone_truncated}.

\item ($C_{CT}$) $v_{k_1}>0,w_{k_1}<0$; $v_{k_2},w_{k_2}>0$; $v_{k_3},w_{k_3} > 0$; ($v_{k_4}, w_{k_4}>0$). \pocta\ is a truncated cone with each face containing one vertex of $K$. The face corresponding to $k_4$ is inactive.

\item ($S$) $v_{k_1}<0,w_{k_1}>0$; $v_{k_3}>0,w_{k_3}<0$, $(v_{k_2},w_{k_2} > 0, v_{k_4}, w_{k_4} > 0)$. \pocta\ is a tilted strip, each side of which contains one vertex of $K$. The faces corresponding to the remaining two vertices are inactive. See Figure \ref{01parocta_split_truncated}.

\item ($S_T$) $v_{k_1}, w_{k_1} > 0$; $v_{k_2}, w_{k_2} > 0$; $v_{k_3} > 0$, $w_{k_3} < 0$; ($v_{k_4}, w_{k_4} > 0$). \pocta\ is a truncated (tilted) strip, each side of which contains a vertex of $K$. The face corresponding to $k_4$ is inactive.
\end{itemize}

\begin{figure}[htbp]
\centering
\begin{tabular}{ccc}
\subfigure[triangle of type C with all vertices outside $K$]{
\includegraphics[width=40mm]{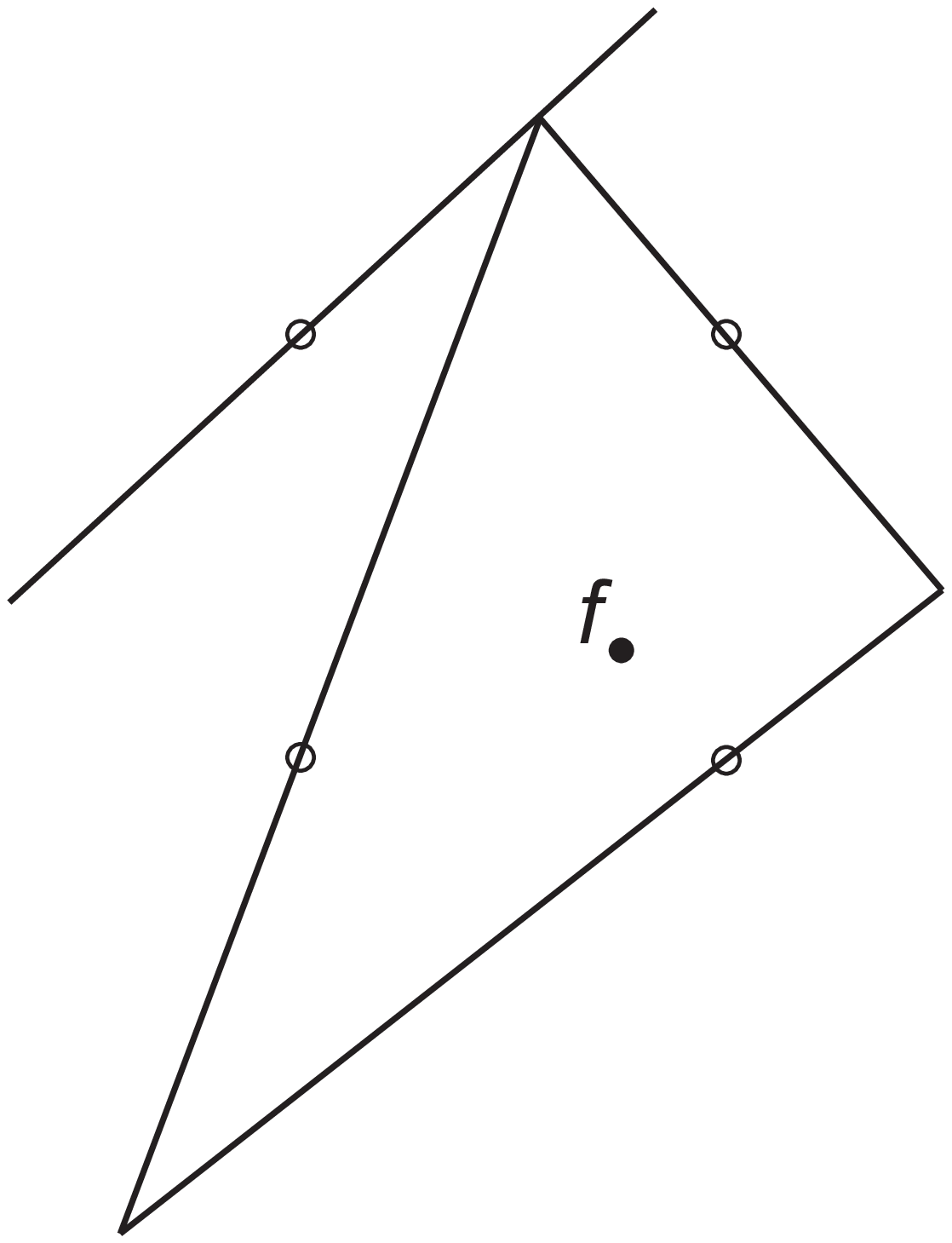}
\label{01parocta_t3t_cut2}
}
\subfigure[triangle of type C with one vertex inside $K$]{
\includegraphics[width=40mm]{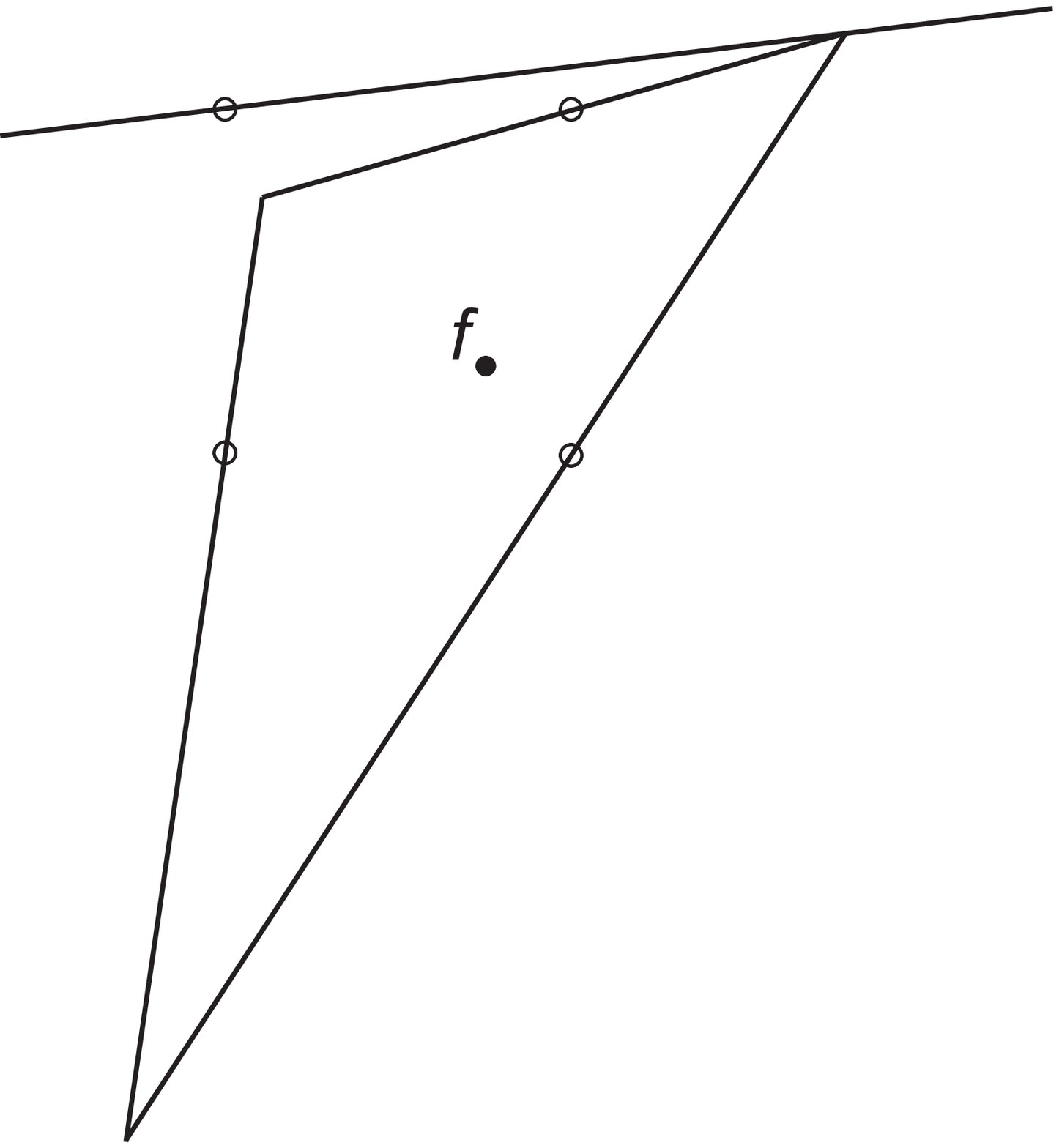}
\label{01parocta_t3t_cut}
}
\subfigure[cone with one face containing two adjacent vertices of $K$]{
\includegraphics[width=40mm]{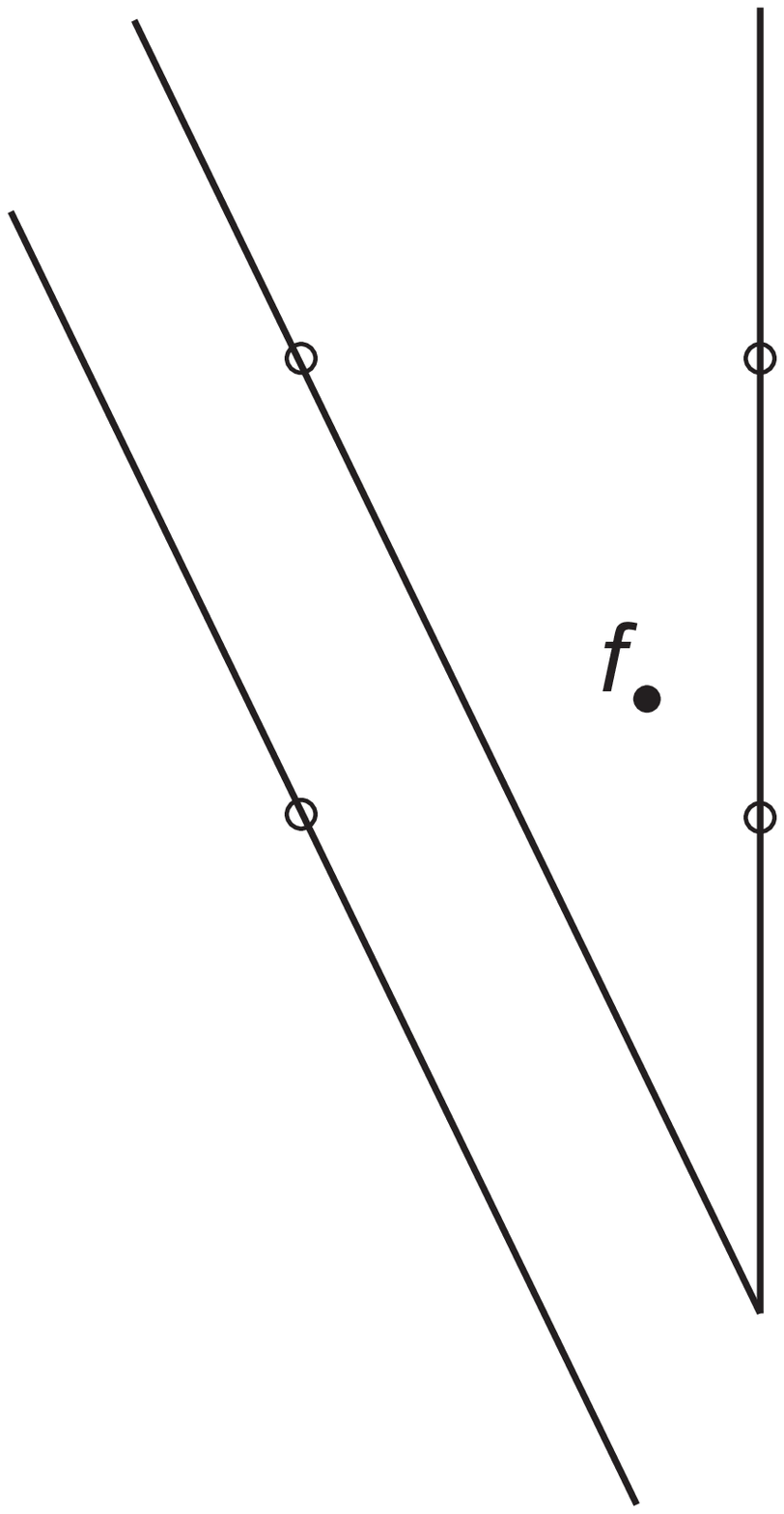}
\label{01parocta_cone_vertical}
}
\\
\subfigure[cone with one face containing two nonadjacent vertices of $K$]{
\includegraphics[width=40mm]{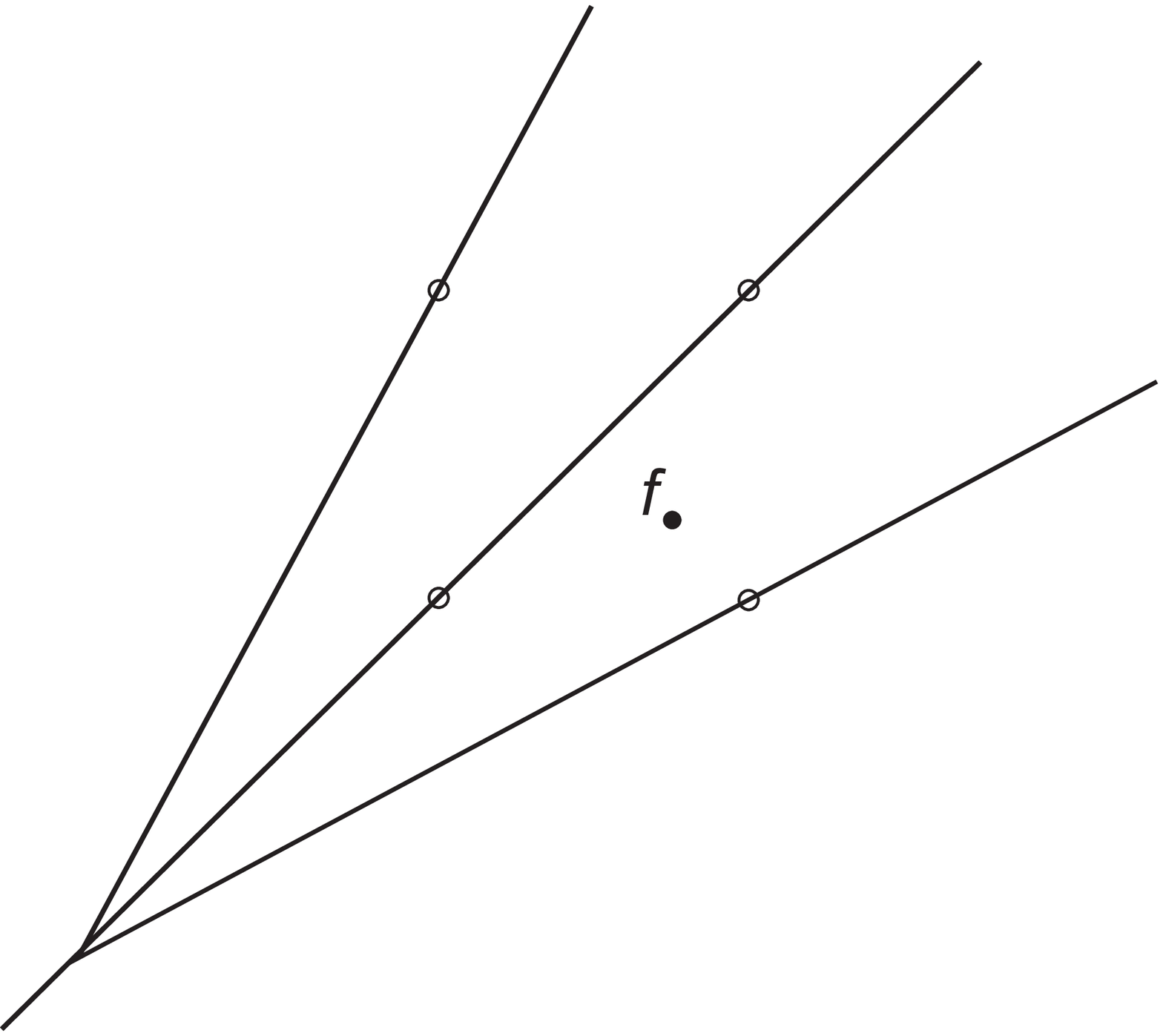}
\label{01parocta_cone_generic}
}
\subfigure[cone with each face containing one vertex of $K$]{
\includegraphics[width=40mm]{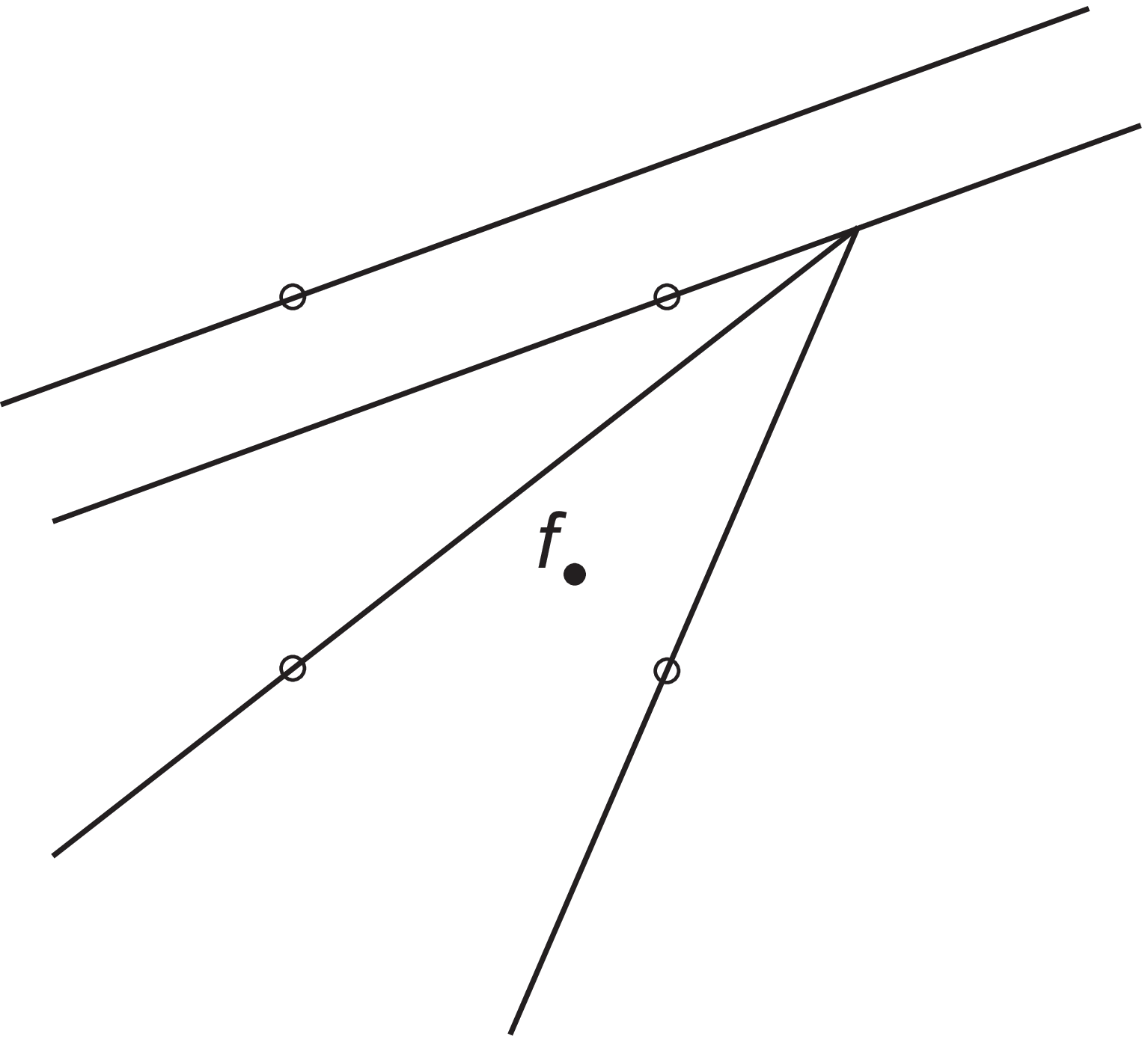}
\label{01parocta_cone_truncated}
}
\subfigure[truncated cone with each face containing one vertex of $K$]{
\includegraphics[width=40mm]{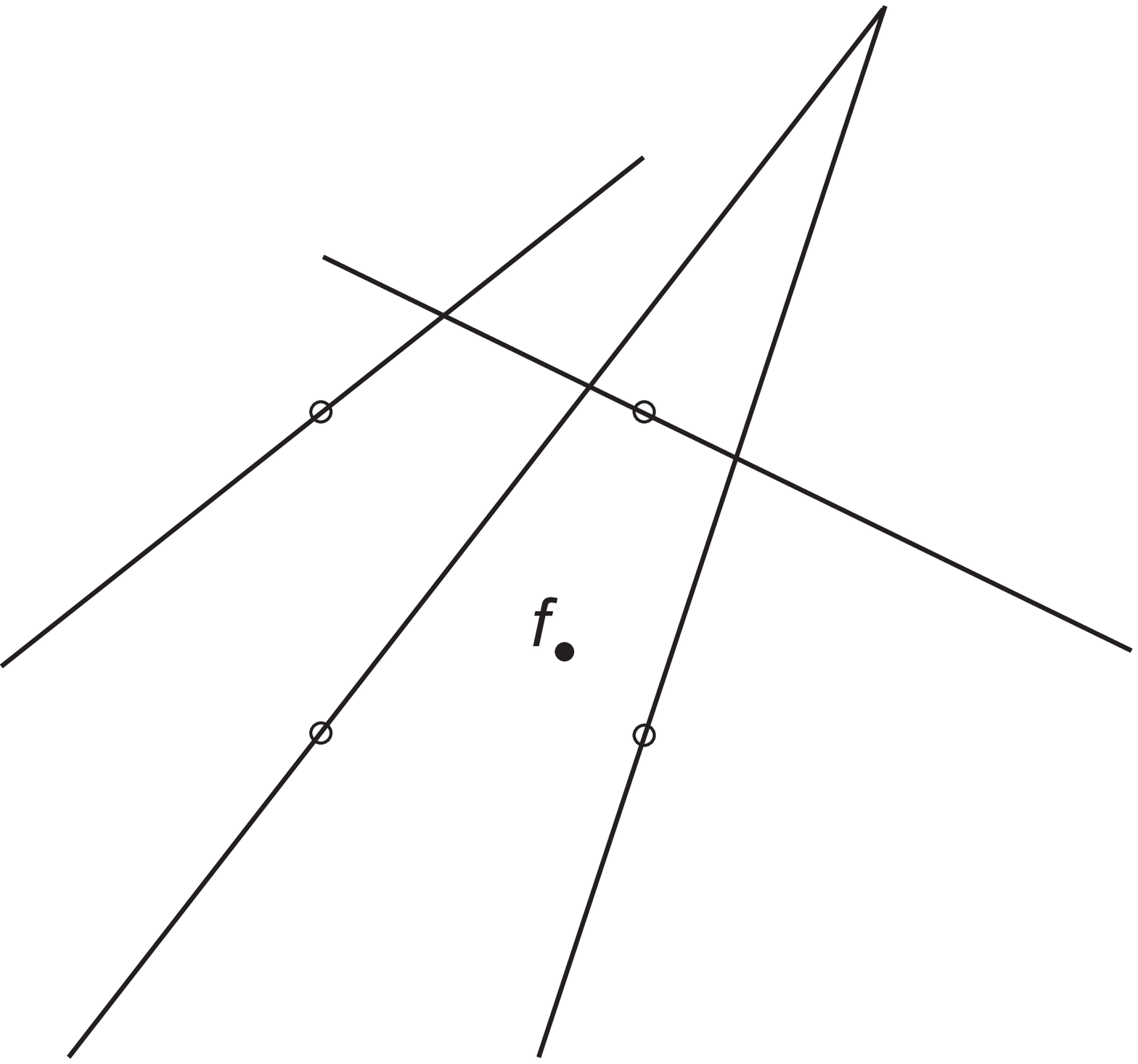}
\label{new_fig_f}
}
\\
\subfigure[tilted strip]{
\includegraphics[width=40mm]{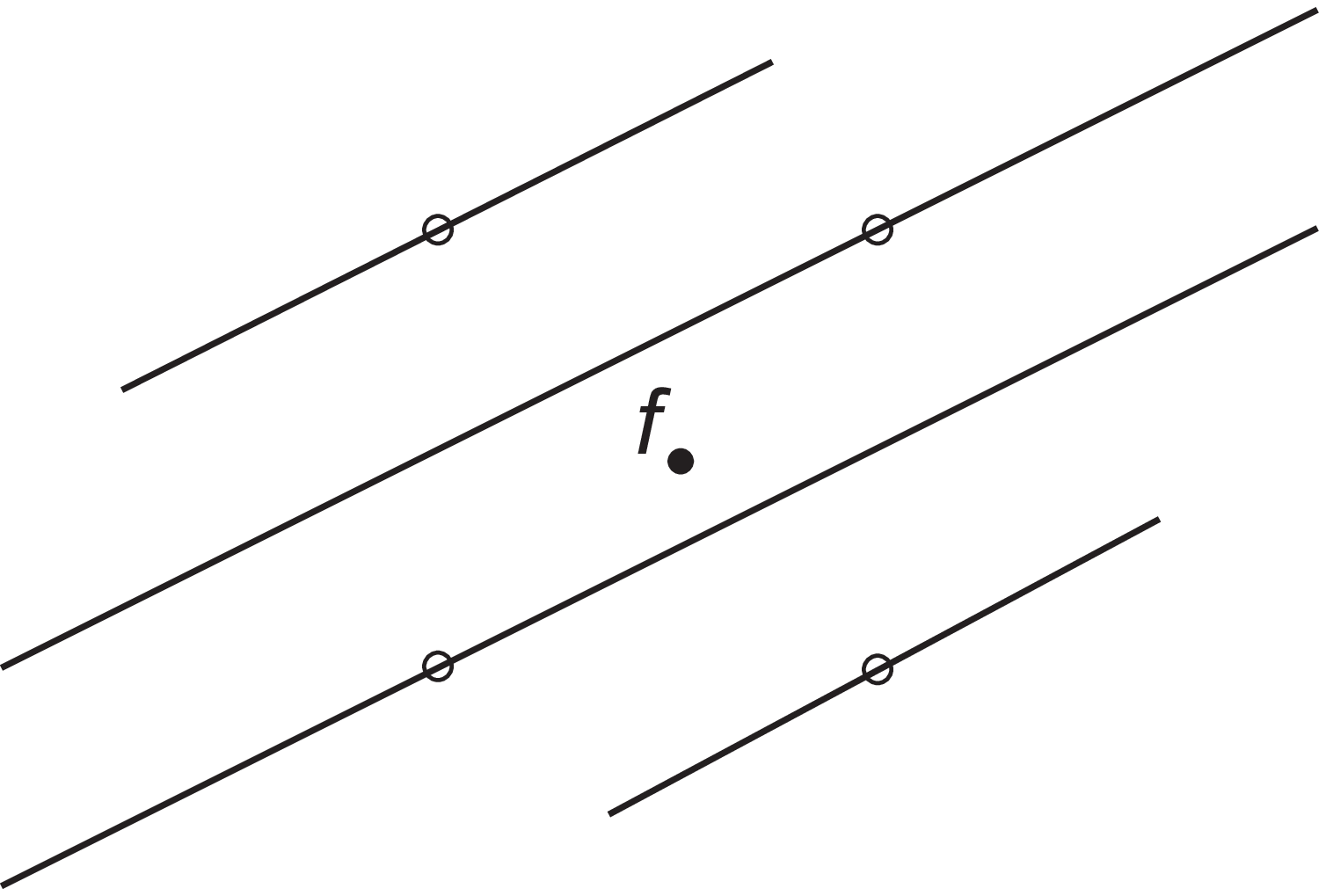}
\label{01parocta_split_truncated}
}
\subfigure[truncated tilted strip]{
\includegraphics[width=40mm]{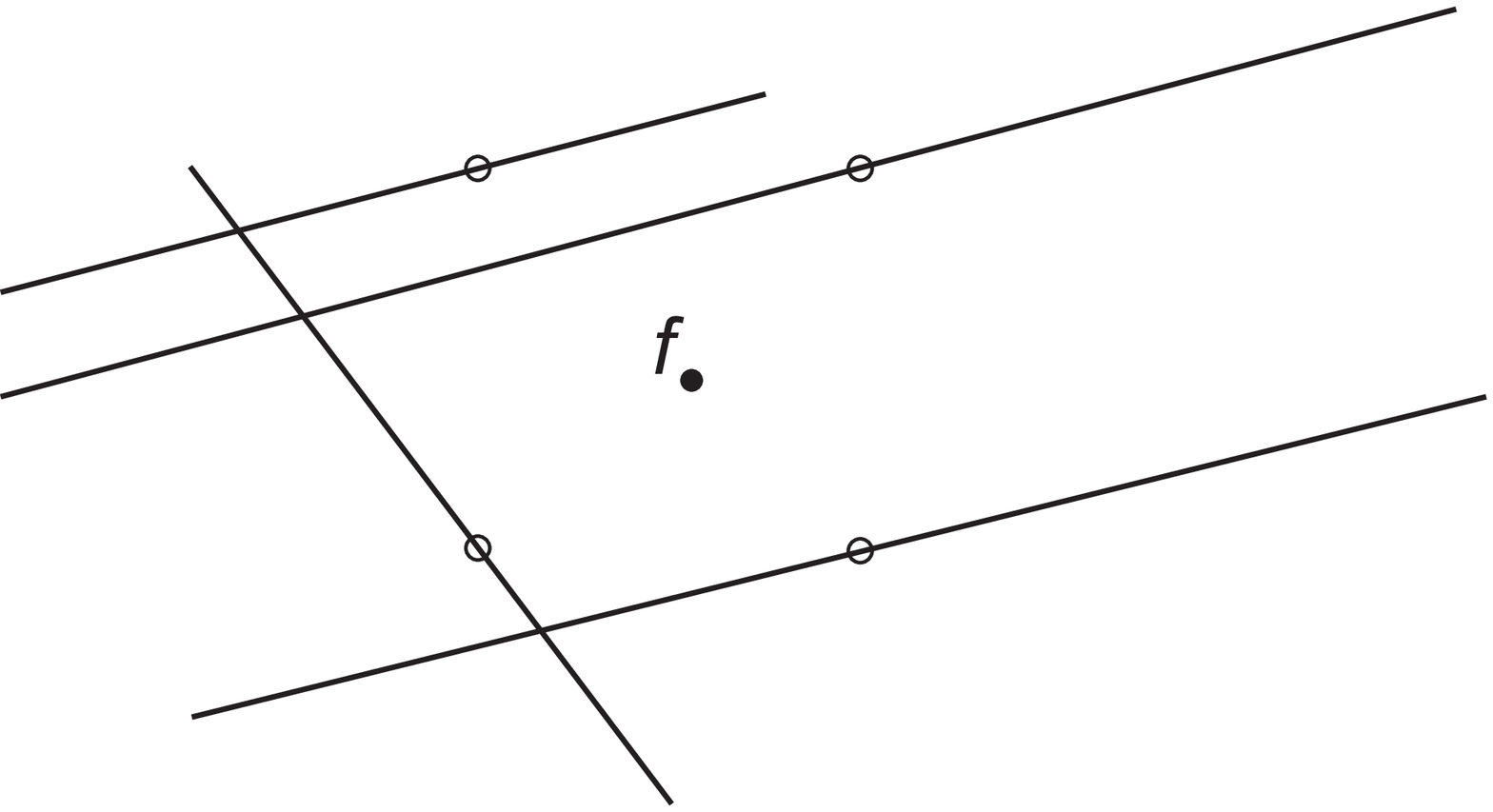}
\label{01parocta_gen_quad_notmaximal}
}
\end{tabular}
\label{01_additional_parocta_configurations}
\caption{Additional configurations of the parametric octahedron for the 0-1 case}
\end{figure}

\begin{example}
Consider the Andersen et al. \cite{ALWW} instance, amended with the condition $x_i\in\{0,1\},i\in \{1,2\}$:
\begin{equation}
\label{ALWW01_example}
\begin{array}{lllllllll}
x_1 & = \frac{1}{4} & + 2  s_1 & + 1  s_2 & - 3  s_3 &		& + 1  s_5\\[6pt]
x_2 & = \frac{1}{2} & + 1 s_1  & + 1  s_2 & + 2  s_3 & - 1  s_4 & - 2  s_5\\[6pt]
\multicolumn{4}{l}{x_1,x_2\in \{0,1\}, \ \ s\geq 0.}
\end{array}
\end{equation}
\end{example}

In section \ref{chapter2_section_case_2_dimensions} we listed the 5 cuts defining the facets of the disjunctive hull for this example, without the 0-1 condition. Using the stronger disjunction expressing the 0-1 condition we obtain the following 12 cuts that define the facets of $\convpde$.
\begin{enumerate}

\item Cut (type $S$): $  2.667 s_1+ 1.333 s_2+ 12 s_3+ 0 s_4+ 1.333 s_5\geq 1 $\\
      $v_1 = 4 ; \ v_2 = 1.333 ; \ v_3 = 1.333 ; \ v_4 = 4 $ \\
      $w_1 = 0 ; \ w_2 = 0 ; \ w_3 = 0 ; \ w_4 = 0 $

\item Cut (type $T_B$): $  2.667 s_1+ 1.333 s_2+ 4.889 s_3+ 0.8889 s_4+ 1.333 s_5\geq 1 $\\
      $v_1 = 2.222 ; \ v_2 = 1.333 ; \ v_3 = 1.333 ; \ v_4 = 0.4444 $ \\
      $w_1 = 0.8889 ; \ w_2 = 0 ; \ w_3 = 0 ; \ w_4 = 1.778 $

\item Cut (type $T_B$): $  2 s_1+ 2 s_2+ 4 s_3+ 1 s_4+ 1.714 s_5\geq 1 $\\
      $v_1 = 2 ; \ v_2 = 1.143 ; \ v_3 = 0 ; \ v_4 = 0 $ \\
      $w_1 = 1 ; \ w_2 = 0.2857 ; \ w_3 = 2 ; \ w_4 = 2 $

\item Cut (type $T_{C1}$): $  2.947 s_1+ 1.053 s_2+ 5.263 s_3+ 0.8421 s_4+ 3.579 s_5\geq 1 $\\
      $v_1 = 2.316 ; \ v_2 = 0.7719 ; \ v_3 = 1.895 ; \ v_4 = 0.6316 $ \\
      $w_1 = 0.8421 ; \ w_2 = 0.8421 ; \ w_3 = -0.8421 ; \ w_4 = 1.684 $

\item Cut (type $T_{C1}$): $  1.63 s_1+ 2.37 s_2+ 8.444 s_3+ 0.4444 s_4+ 1.926 s_5\geq 1 $\\
      $v_1 = 3.111 ; \ v_2 = 1.037 ; \ v_3 = -0.7407 ; \ v_4 = 2.222 $ \\
      $w_1 = 0.4444 ; \ w_2 = 0.4444 ; \ w_3 = 3.111 ; \ w_4 = 0.8889 $

\item Cut (type $T_{C2}$): $  4.364 s_1+ 2.545 s_2+ 3.273 s_3+ 1.091 s_4+ 0.3636 s_5\geq 1 $\\
      $v_1 = 1.818 ; \ v_2 = 1.818 ; \ v_3 = 1.818 ; \ v_4 = -0.3636 $ \\
      $w_1 = 1.091 ; \ w_2 = -0.7273 ; \ w_3 = 0.7273 ; \ w_4 = 2.182 $

\item Cut (type $T_{C2}$): $  3.765 s_1+ 3.059 s_2+ 2.588 s_3+ 1.176 s_4+ 0.7059 s_5\geq 1 $\\
      $v_1 = 1.647 ; \ v_2 = 1.647 ; \ v_3 = 0.7059 ; \ v_4 = -0.7059 $ \\
      $w_1 = 1.176 ; \ w_2 = -0.4706 ; \ w_3 = 2.353 ; \ w_4 = 2.353 $

\item Cut (type $C_A$): $  12 s_1+ 8 s_2+ 12 s_3+ 0 s_4  -4 s_5\geq 1 $\\
      $v_1 = 4 ; \ v_2 = 4 ; \ v_3 = 4 ; \ v_4 = 4 $ \\
      $w_1 = 0 ; \ w_2 = -4 ; \ w_3 = 4 ; \ w_4 = 0 $

\item Cut (type $C_B$): $  32 s_1+ 20 s_2  -20 s_3+ 4 s_4+ 12 s_5\geq 1 $\\
      $v_1 = -4 ; \ v_2 = 4 ; \ v_3 = 4 ; \ v_4 = -12 $ \\
      $w_1 = 4 ; \ w_2 = 4 ; \ w_3 = -4 ; \ w_4 = 8 $

\item Cut (type $C_B$): $  12 s_1+ 8 s_2+ 44 s_3  -4 s_4  -4 s_5\geq 1 $\\
      $v_1 = 12 ; \ v_2 = 4 ; \ v_3 = 4 ; \ v_4 = -4 $ \\
      $w_1 = -4 ; \ w_2 = -4 ; \ w_3 = 4 ; \ w_4 = 4 $

\item Cut (type $C_C$): $  -2 s_1+ 6 s_2+ 52 s_3+ 2 s_4+ 4 s_5\geq 1 $\\
      $v_1 = 0 ; \ v_2 = 0 ; \ v_3 = -8 ; \ v_4 = 8 $ \\
      $w_1 = 2 ; \ w_2 = 2 ; \ w_3 = 14 ; \ w_4 = -2 $

\item Cut (type $C_B$): $  8 s_1  -4 s_2+ 12 s_3+ 16 s_4+ 44 s_5\geq 1 $\\
      $v_1 = 4 ; \ v_2 = -9.333 ; \ v_3 = 12 ; \ v_4 = 4 $ \\
      $w_1 = 0 ; \ w_2 = 16 ; \ w_3 = -16 ; \ w_4 = 0 $

\end{enumerate}

The above list of 12 cuts includes 3 of the 5 cuts defining facets of $\convpd$, namely 1, 2 and 4, which appear on our list in position 3, 2 and 1, respectively. The remaining 2 facets of $\convpd$, given by cuts 3 and 5, are redundant for $\convpde$; namely, cut 3 is a convex combination of cuts 2, 3, 6 and 7 on our list, while cut 5 is a convex combination of cuts 1 and 5 on our list.

The number of facets of $\convpde$ substantially exceeds the number of facets of $\convpd$. In order to assess the impact of the two sets of cuts, we computed the average integrality gap for 1,000 randomly generated objective functions. Adding the 5 cuts valid for the 2-row MIP reduces this gap by 77\%; while adding the additional cuts valid for the 0-1 case reduces 100\% of the gap.

Next we discuss the strengthening of valid cuts for $\pde$ when some variables $\sj$ are integer-constrained. Let $J_1$ be the index set of such variables.

First of all, we observe that the standard modularization procedure described in Theorem~\ref{th7.2} for strengthening cuts for $\PD$ is not valid in the case of cuts for $\pde$. Indeed, Lemma~\ref{lem2.3.5} which underlies the correctness of the procedure in the case of $\PD$, is no longer valid in the case of $\pde$: if the disjunction (\ref{nonbasicdisjunction_integralnonbasics}) is modified by replacing every inequality with equality, then it is no longer equivalent to the disjunction obtained by replacing $\rij$ with $\rij - \mij$. Instead, we will use a different modularization, known in the literature under the name of monoidal strengthening \cite{BalasJeroslow}.

Consider a disjunction of the form
\begin{equation}\label{eq(1)}
\bigvee_{k \in Q} ( A^k x \ge a^k_0), \quad A^k = (a^k_j), \ j \in J, \quad a^k_j \in \R^m, \ j \in J \cup \{0\},
\end{equation}
and the valid cut $\al x \ge 1$, where
\begin{equation}\label{eq(2)}
\al_j = \max_{k \in Q} \{ \theta^k a^k_j / \theta^k a^k_0 \}
\end{equation}
for some $\thk \in \R^m_+$, $k \in Q$.

Suppose now that for each $\Ak x$, $k \in Q$, a lower bound  $\bko \le \ako$ is known, i.e. $\Ak x \ge \bko$, $k \in Q$.

\begin{theorem}\label{th8.1}
Let $M := \{ m \in \Z^{|Q|} : \sum_{k \in Q} m^k \ge 0\}$. If $x_j \in \Z$, $j \in J_1$, then the cut $\al x \ge 1$ can be strengthened to $\bal x \ge 1$, where
\begin{equation}\label{eq(3)}
\bal_j = \min_{m \in M} \max_{k \in Q} \left\{ \left( \thk \akj + \mkj \thk ( \ako - \bko)\right)/\thk \ako\right\} \quad\quad j \in J_1
\end{equation}
and $\bal_j = \al_j$ for $j \in J \sm J_1$.
\end{theorem}

\begin{proof}
See \cite{BalasJeroslow} or \cite{BalasDisjunctiveProgramming}.
\end{proof}

We will now apply this Theorem to our case, first with the disjunction (\ref{nonbasicdisjunction_integralnonbasics}), then with the stronger disjunction defining $\pde$. Let $\al x \ge 1$ be an inequality implied by the disjunction (\ref{nonbasicdisjunction_integralnonbasics}), i.e. a valid inequality for $\PD$, and let's assume that $x_i \in \{0,1\}$, $i = 1,2$. It is not hard to see that a lower bound on the lefthand side of each of the 8 inequalities that occur in (\ref{nonbasicdisjunction_integralnonbasics}) is obtained by subtracting 1 from the righthand side. This means that if $\ako$ denotes the righthand side and $\bko$ the lower bound on the lefthand side of the $k$-th term, then $\ako - \bko = {1 \choose 1}$.

Now the cut from the disjunction (\ref{nonbasicdisjunction_integralnonbasics}) is $\al s \ge 1$, where
$$\al_j = \max_{k \in \{1,\ldots,4\}} \{ \alkj \},
$$
and
\begin{equation}\label{eq(4)}
\begin{array}{rcl}
\alj^1 & = & - \roj \bv_1 - \rtj \bw_1,\\[6pt]
\alj^2 & = & ~~\,\roj \bv_2 - \rtj \bw_2,\\[6pt]
\alj^3 & = & ~~\, \roj \bv_3 + \rtj \bw_3,\\[6pt]
\alj^4 & = & - \roj \bv_4 + \rtj \bw_4.
\end{array}
\end{equation}
To apply the theorem to this case, notice that $\thk = ( \bv_k,\bw_k)$ and $\thk \akj = \alkj$, $\thk \ako = 1$ for $k = 1,\ldots,4$.

\begin{corollary}\label{co8.2}
Let $x_1,x_2 \in \{ 0,1\}$, and $M = \{ m \in \Z^4 : \sum^4_{k=1} m^k \ge 0\}$. Then $\bal s \ge 1$ is a valid cut for $\PD$, with $\bal_j = \max_{k \in \{ 1,\ldots,r \}} \{ \balkj \}$, and
$$\begin{array}{rcl}
\balkj & = & \left\{
\begin{array}{ll}\displaystyle
\min_{\mkj \in M} \max_{k \in \{ 1,\ldots,4 \}} \{ \alkj + \mkj (\bv_k + \bw_k)\} \quad & j \in J_1\\[12pt]
\alkj & j \in J \sm J_1
\end{array}
\right.
\end{array}
$$
\end{corollary}

\begin{proof}
Denoting
\begin{equation}\label{eq(5)}
\begin{array}{llll}
\displaystyle\ao^1 = {f_1 \choose f_2}, \quad & \displaystyle\ao^2 = {1-f_1 \choose f_2}, \quad & \displaystyle\ao^3 = {1-f_1 \choose 1-f_2}, \quad & \displaystyle\ao^4 = {f_1 \choose 1-f_2},\\[18pt]
\displaystyle \bo^1 = {f_1 - 1 \choose f_2-1}, & \displaystyle\bo^2 = {-f_1 \choose f_2-1}, & \displaystyle\bo^3 = {-f_1 \choose -f_2}, & \displaystyle\bo^4 = {f_1-1 \choose -f_2},
\end{array}
\end{equation}
it is easy to see that for $k = 1,\ldots,4$,
$$(\bv_k , \bw_k)(\ako - \bko) = \bv_k + \bw_k.
$$
\end{proof}

We now turn to strengthening a valid inequality for $\pde$, the set defined by the disjunction (\ref{nonbasicdisjunction_integralnonbasics}$^=$), obtained from (\ref{nonbasicdisjunction_integralnonbasics}) by replacing each inequality with equality. In this case the cut from (\ref{nonbasicdisjunction_integralnonbasics}$^=$) is $\tal x \ge 1$, where $\tal_j = \max_{k \in \{ 1,\ldots,4 \}} \tal^k_j$ and the $\tal^k_j$ are given by the same expressions (\ref{eq(4)}) as $\alkj$, with the important difference that the parameters $(\bv_k,\bw_k)$ are unrestricted in sign. However, from the normalization constraints (\ref{cglpmip_normalized}) it follows that for any $k \in \{ 1,\ldots,4 \}$, at most one member of the pair $(\bv_k,\bw_k)$ can be negative.

In order to derive the lower bounds $\bko$ required by Theorem~\ref{th8.1}, the best way is to represent each equation of (\ref{nonbasicdisjunction_integralnonbasics}$^=$) as a pair of inequalities; i.e. the first term of (\ref{nonbasicdisjunction_integralnonbasics}$^=$) is restated as
\begin{equation}\label{eq(6)}
\left(
\begin{array}{rcl}
-r^1s & \ge & ~~\,f_1\\[6pt]
r^1s & \ge & -f_1\\[6pt]
-r^2s & \ge & ~~\,f_2\\[6pt]
r^2s & \ge & -f_2
\end{array}
\right)
\end{equation}
and so on. Denoting the corresponding parameters or multipliers by $\vk', \vk'', \wk', \wk''$ for $k = 1,\ldots,4$, we see that since at most one of the pairs of inequalities corresponding to an equation can be active in any given solution, at most one member of each pair $(\vk', \vk'')$ can be positive, and the same holds for each pair $(\wk',\wk'')$. Furthermore, it becomes clear that if in the equality formulation \sevene\ a parameter, say $v_1$, takes on a negative value $\bv_1 < 0$ in a solution, this corresponds to the fact that the member of the pair of inequalities corresponding to the equation associated with $v_1$ that is active, is the one with $\le$, i.e. with the inequality reversed.

\begin{corollary}\label{co8.3}
Let $M = \{ m \in \Z^4 : \sum^4_{k=1} m^k \ge 0\}$. Then $\hal s \ge 1$ is a valid cut for $\pde$, with $\hal_j = \max_{k \in \{ 1,\ldots,4 \}} \{ \hal^k_j \}$, and
$$\begin{array}{rcl}
\hal^k_j & = & \left\{
    \begin{array}{ll}
    \displaystyle\min_{\mkj \in M} \max_{k \in \{ 1,\ldots,r \}} \{ \tal^k_j + m^k_j ( \bv^+_k + \bw^+_k)\}, \quad \quad & j \in J_1\\[12pt]
    \tal^k_j & j \in J \sm J_1
    \end{array}
\right.
\end{array}
$$
where $\bvkp = \max \{ \bvk,0\}$ and $\bwkp = \max \{ \bwk,0\}$.
\end{corollary}

\begin{proof}
If, using the inequality formulation (\ref{eq(6)}) of the disjunction \sevene, we denote the righthand sides of the four terms by
\begin{equation}\label{eq(7)}
\ta^1_0 = \left(
    \begin{array}{c}
    ~~\,\fo \\ -\fo \\ ~~\,\ft \\ -\ft
    \end{array}
    \right), \quad
\ta^2_0 = \left(
    \begin{array}{c}
    1-\fo \\ \fo - 1 \\ ~~\,\ft \\ -\ft
    \end{array}
    \right), \quad
\ta^3_0 = \left(
    \begin{array}{c}
    1-\fo \\ \fo - 1 \\ 1-\ft \\ \ft-1
    \end{array}
    \right), \quad
\ta^4_0 = \left(
    \begin{array}{c}
    ~~\,\fo \\ -\ft \\ 1-\ft \\ \ft-1
    \end{array}
    \right),
\end{equation}
then the lower bounds on the expressions on the lefthand sides of the inequalities are no longer equal to the righthand side minus 1. Instead, we have the following situation:
\begin{equation}\label{eq(9)}
\tb^1_0 = \left(
    \begin{array}{c}
    \fo - 1 \\ -\fo \\ \ft-1 \\ -\ft
    \end{array}
    \right), \quad
\tb^2_0 = \left(
    \begin{array}{c}
    -\fo \\ \fo - 1 \\ \ft - 1 \\ -\ft
    \end{array}
    \right), \quad
\tb^3_0 = \left(
    \begin{array}{c}
    -\fo \\ \fo - 1 \\ -\ft \\ \ft-1
    \end{array}
    \right), \quad
\tb^4_0 = \left(
    \begin{array}{c}
    \fo - 1 \\ -\fo \\ -\ft \\ \ft-1
    \end{array}
    \right).
\end{equation}
As a consequence,
$$\ta^k_0 - \tb^k_0 = \left( \begin{array}{c} 1 \\ 0 \\ 1 \\ 0\end{array}\right) \quad\mbox{ for } k = 1,\ldots,4
$$
Thus, if we denote $\bvkp : = \max \{ \bvk,0 \}$, $\bwkp = \max \{ \bwk,0 \}$, we have $(\bvkp,\bwkp)(\ako,\bko) = (\bvkp,\bwkp)$, $k= 1,\ldots,4$, and the expression for $\hal^k_j$ follows.
\end{proof}

Finding the optimal $\mkj \in M$ requires a small (single digit) number of comparisons. While \cite{BalasJeroslow} and \cite{BalasDisjunctiveProgramming} give simple procedures for the case of a general disjunction, the optimal $\mkj$ of Corollary~\ref{co8.3} for a given $j \in J_1$ can be found as follows:
\begin{itemize}
\item Start with $\mkj = 0$ for all $k$ and apply the \\{\em Iterative Step}.

    \begin{itemize}
    \item[{\small $\bullet$}] Find $\almaxj = \max_k \alkj$, $\alminj = \min_k \alkj$ and let $\mjmax, \mjmin$ be the corresponding values of $\mkj$.

    \item[{\small $\bullet$}] Set $\mjmax = \mjmax - t$, $\mjmin = \mjmin + t$, where $t$ is the smallest positive integer for which the identity of $\almaxj$ changes.

    \item[{\small $\bullet$}] If the value of $\max_k \alkj$ has not been reduced, stop with $\mjmax = \mjmax - t + 1$, $\mjmin = \mjmin + t - 1$, and $\mkj$ unchanged for $k \ne \max, \min$. Otherwise repeat.
    \end{itemize}
\end{itemize}

In the case where \pocta\ is a triangle with each face containing exactly one vertex of $K$, the term of the disjunction \sevene\ corresponding to the vertex of $K$ left outside the triangle plays no role in defining the cut, hence it can be dropped and the strengthening becomes simpler. This is even more true of the case of a cone, where only two terms of the disjunction are active. A particularly simple case is that of a ``fixed shape'' cone with apex at a vertex of $K$, and one face containing a side of $K$, the other face containing the diagonal of $K$. There are eight such cones, and every fractional $(\fo,\ft) \ne (\frac{1}{2},\frac{1}{2})$ (i.e. not lying on the diagonal of $K$) is strictly contained in four of them (see figure~\ref{fixed_cones_figure} in the next section).

We will illustrate the monoidal strengthening procedure on the conic cuts obtainable from these disjunctions. Here is a couple of them:
\begin{enumerate}
\item $(-x_2 \ge 0) \vee (-x_1 + x_2 \ge 0)$
\item $(x_2 \ge 1) \vee (-x_1 - x_2 \ge -1)$
\end{enumerate}
or, after substituting $f_i + r^is$ for $x_i$, $i = 1,2$,
\begin{enumerate}
\item $(-r^2s \ge f_2) \vee (( - r^1 + r^2)s \ge \fo - \ft)$
\item $(r^2s \ge 1-\ft) \vee ((-r^1 - r^2)s \ge \fo + \ft - 1)$.
\end{enumerate}
Each disjunction violated by the point $(\fo,\ft)$ has positive righthand sides and gives rise to a valid cut $\al s \ge 1$, with coefficients $\al_j$ shown below, obtained by using multipliers normalized to yield a cut with a righthand side of 1:
\begin{enumerate}
\item $\max \left\{ \frac{-r^2_j}{\ft}, \frac{-r^1_j + r^2_j}{\fo - \ft}\right\}$
\item $\max \left\{ \frac{r^2_j}{1-\ft}, \frac{-r^1_j - r^2_j}{\fo + \ft - 1}\right\}$
\end{enumerate}

To apply the strengthening procedure, we note that for each of the 16 terms of the above 8 disjunctions, the lower bound on the lefthand side of the inequality is just 1 unit less than the righthand side, hence the difference between the latter and the former is exactly 1. Further, the weights $(\bv_k,\bw_k)$ are normalized so that $\bvk + \bwk = 1$, $k = 1,2$. The resulting strengthened coefficients for the above illustration are
\begin{enumerate}
\item $\min_{\mkj \in M} \max \left\{ \frac{-r^1_j + m^1_j}{\ft}, \frac{r^1_j - r^2_j + m^2_j}{\fo - \ft} \right\}$
\item $\min_{\mkj \in M} \max \left\{ \frac{r^2_j + m^1_j}{1-\ft}, \frac{-r^1_j - r^2_j + m^2_j}{\fo + \ft - 1}\right\}$
\end{enumerate}

\section{Computational Experiments}\label{chapter2_section_computational_results}

In this section we present computational experiments with cuts derived from fixed configurations of the parametric octahedron. We assess the strength of the cuts by analyzing the gap closed on instances from MIPLIB3\_C\_V2 \cite{MIPLIB3CV2} when used in combination with standard Gomory cuts. MIPLIB3\_C\_V2 is a collection of 68 instances by Margot which are slight variations of the standard MIPLIB3 \cite{MIPLIB3} and for which the validity of a candidate solution can be checked in finite precision arithmetic.
We restricted the collection to a subset of 41 instances. The considered instances are such that they contain at least 2 binary variables fractional in the optimal LP solution and the cut generation procedure on each round takes less than 3600 seconds.

We generated the following two families of cuts
\begin{itemize}
\item Cuts from 4 Triangles $T_A$ (shown in Figure \ref{fixed_t1t_figure}) whose vertices, expressed in terms of their $x_1,x_2$ coordinates, are:
  \begin{itemize}
  \item $(0,0);(2,0);(0,2)$
  \item $(-1,0);(1,0);(1,2)$
  \item $(0,-1);(2,1);(0,1)$
  \item $(1,-1);(1,1);(-1,1)$
  \end{itemize}
\item Cuts from 4 of the 8 cones of type $C_A$ (shown in Figure \ref{fixed_cones_figure}):
  \begin{itemize}
  \item apex at $(0,0)$ and rays $(1,0),(1,1)$
  \item apex at $(0,0)$ and rays $(0,1),(1,1)$
  \item apex at $(0,1)$ and rays $(1,0),(1,-1)$
  \item apex at $(0,1)$ and rays $(0,-1),(1,-1)$
  \item apex at $(1,1)$ and rays $(-1,0),(-1,-1)$
  \item apex at $(1,1)$ and rays $(0,-1),(-1,-1)$
  \item apex at $(1,0)$ and rays $(-1,0),(-1,1)$
  \item apex at $(1,0)$ and rays $(0,1),(-1,1)$
  \end{itemize}
\end{itemize}
The reason we only used 4 of these 8 cones is that every $(\fo,\ft)$-pair is contained in 4 of these 8 cones.

\begin{figure}[ht!]
\centering
\includegraphics[width=90mm]{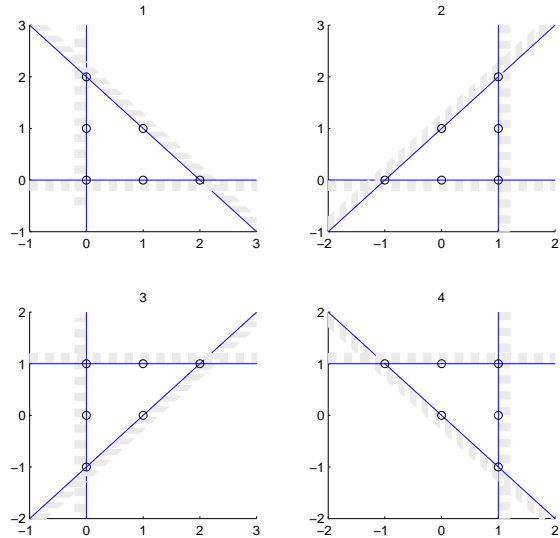}
\caption{Fixed shape Triangles $T_A$}
\label{fixed_t1t_figure}
\end{figure}
\begin{figure}[ht!]
\centering
\includegraphics[width=160mm]{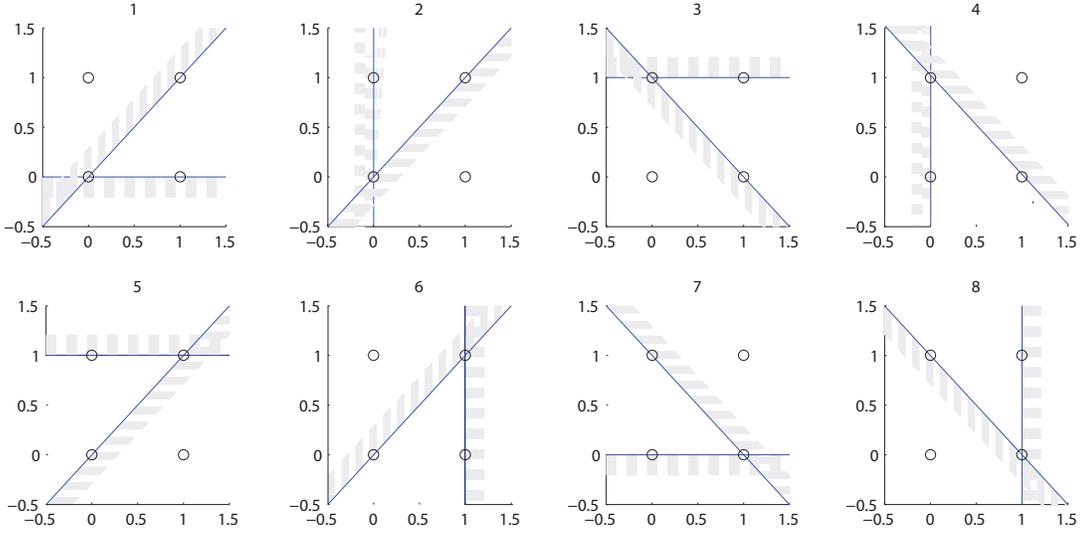}
\caption{Fixed shape cones $C_A$}
\label{fixed_cones_figure}
\end{figure}

For each instance, we first solved the linear programming relaxation and generated a round of Gomory mixed integer (GMI) cuts, a round being one cut from every row of the optimal simplex tableau associated with a binary basic variable with a fractional value. We then generated from each pair of rows with at least one fractional binary basic variable either (a) all cuts from the 4 triangles $T_A$, or (b) all cuts from 4 of the 8 cones $C_A$ or both, and strengthened them via standard modularization (in case (a)) or monoidal strengthening (in case (b)).

We call this cut generating cycle a {\em round}. At the end of each round, we reoptimized the resulting linear program and removed all cuts that were not tight at the optimum. We generated up to 5 rounds of cuts for each instance. A statement of our routine follows.

\begin{codebox}
\Procname{$\proc{\textit{Cut Generating Procedure}}(r,f)$}
\li Solve LP relaxation $P$
\li \For $k \gets 1$ {\bf up to} $5$
\li \Do
\li      Initialize cut collection $C \gets$ empty
\li      \For each binary basic $x_i$ fractional in the \\ \quad\quad \quad\quad current solution
\li      \Do
\li           Compute Gomory cut $G_i$
\li           $C \gets C \cup G_i$
         \End
\li      \For each binary basic pair $x_i,x_j$ with at \\ \quad\quad\quad\quad  least one fractional in the current solution
\li      \Do
\li          \If GenerateTriangles$==$true
\li          \Then generate the cuts $T^1_{ij},\dots,T^4_{ij}$ from each of the 4 triangles of \\ \quad\quad \quad\quad\quad\quad\quad\quad\ \,  $T_A$ that contain the fractional solution in their interior
\li		  \If StrengthenCuts$==$true
\li		  \Then strengthen the cuts $T^1_{ij},\dots,T^4_{ij}$ via standard modularization
                  \End
\li          $C \gets C \cup T^1_{ij},\dots,T^4_{ij}$
	     \End
\li          \If GenerateCones$==$true
\li          \Then generate the cuts $K^1_{ij},\dots,K^8_{ij}$ from each of the 8 cones $C_A$ that \\ \quad\quad \quad\quad\quad\quad\quad\quad\ \,   contain the fractional solution in their interior
\li		  \If StrengthenCuts$==$true
\li		  \Then strengthen the cuts $K^1_{ij},\dots,K^8_{ij}$ via monoidal cut \\ \quad\quad \quad\quad\quad\quad \quad\quad \quad\quad\quad\quad\,   strengthening
                  \End
\li          $C \gets C \cup K_1,\dots,K^8_{ij}$
	     \End
         \End
\li Resolve $P$ and get new solution $\bar{x}^k$ with value $\overline{opt}^k$
\li Remove from $P$ the cuts in $C$ that are not tight at $\bar{x}^k$
    \End
\end{codebox}

The Gomory mixed integer (GMI) cut generator we used is the \emph{CglGomory} routine of the Cgl package of COIN-OR \cite{COIN-OR}. Tables~\ref{tab9.1}--\ref{tab9.4} summarize the results of our experiments with these cuts. Table~\ref{tab9.1} shows the outcome of applying all three types of cuts in the above described manner, with strengthening, for one round. Column 1 lists the 41 test instances mentioned above. Column 2 shows the percentage of the integrality gap closed by one round of GMI cuts, while the next two columns show the number of cuts generated and added to the LP relaxation, along with the number of cuts deleted after reoptimization as nonbinding. The next three columns show the same data (i.e. percentage of gap closed and number of cuts added, respectively deleted) after generating a cut from each of the 4 triangles $T_A$ associated with every pair of basic 0-1 variables with at least one fractional member, and a cut from each of the 8 cones associated with every such pair, provided the cone contains such a pair in its interior. Finally, the last column shows the percentage improvement in the integrality gap closed by all three types of cuts versus the GMI cuts alone.

As the table shows, the integrality gap closed, which is 19.49\% in the case of the GMI cuts, reaches 29.06\% when the two remaining types of cuts are added, an increase of 49.14\%. The number of triangle cuts and conical cuts generated is of course much larger than that of GMI cuts. While the latter is bounded by the number of basic 0-1 variables fractional at the optimum, in case of the other two types of cuts this number gets multiplied by 8 times the number of basic 0-1 variables, fractional or not. From the table it is clear that after reoptimization few of the added cuts remain active (about 2\%), while the rest get removed. The table also reveals marked differences in the impact of the 2-row cuts on different instances, from 0 impact in about 40\% of the instances, to a more than 7-fold increase of the gap closed in the highest-impact case.

Tables~\ref{tab9.2}--\ref{tab9.3} show the effect of using only triangle cuts or only conic cuts on top of the GMI cuts. Clearly, the joint effect of using both types of cuts is substantially stronger than is the case with a single type.

Finally, Table~\ref{tab9.4} shows the effect of generating both types of 2-row cuts on top of GMI cuts, as in Table~\ref{tab9.1}, but this time for 5 rounds instead of just 1. The improvement in gap closing keeps growing after every round. At the end of the 5 rounds, the gap closed is 38.78\%, roughly twice as large as the 19.49\% gap closure obtained by 1 round of GMI cuts.\clearpage

\setcounter{table}{0}
\begin{small}
\begin{table}
\caption{GMI cuts  + Triangle cuts + Conic cuts,  all strengthened, 1 round}
\label{tab9.1}
\begin{center}
\footnotesize
\begin{tabular}{ l | r r r D D | r r r D D | r}	\cline{2-11}
  & \multicolumn{5}{c|}{GMI} & \multicolumn{5}{c|}{GMI+ $T_A$ + $C_A$, strengthened} &  \\\cline{2-11}
  & \multicolumn{1}{c}{Gap} & \multicolumn{1}{c}{Cuts} & \multicolumn{1}{c}{Cuts} & &  & \multicolumn{1}{c}{Gap} & \multicolumn{1}{c}{Cuts} & \multicolumn{1}{c}{Cuts} & & & \\
  & \multicolumn{1}{c}{closed} & \multicolumn{1}{c}{added} & \multicolumn{1}{c}{deleted} & &  & \multicolumn{1}{c}{closed} & \multicolumn{1}{c}{added} & \multicolumn{1}{c}{deleted} & & & Improvement \\
Instance  & \multicolumn{1}{c}{\%} & \multicolumn{1}{c}{\#} & \multicolumn{1}{c}{\#} & &  & \multicolumn{1}{c}{\%} & \multicolumn{1}{c}{\#} & \multicolumn{1}{c}{\#} & & & \multicolumn{1}{c}{\%}\\\hline
air03	& 100	& 36	& 10	& 0.06	& 0.13	& 100	& 12744	& 12667	& 40.11	& 96.23	& 0.00 \\
cap6000	& 41.65	& 2	& 1	& 0.02	& 0.03	& 41.65	& 1360	& 1359	& 2.33	& 7.2	& 0.00 \\
danoint	& 0.26	& 24	& 13	& 0.02	& 0.02	& 0.26	& 24	& 13	& 0.02	& 0.53	& 0.00 \\
dcmulti	& 45.75	& 49	& 12	& 0.01	& 0.01	& 45.86	& 392	& 362	& 0.04	& 0.32	& 0.24 \\
egout	& 21.84	& 16	& 0	& 0	& 0	& 60.91	& 3461	& 3423	& 0.02	& 0.06	& 178.89 \\
enigma	& 100	& 6	& 5	& 0	& 0	& 100	& 341	& 340	& 0.02	& 0.03	& 0.00 \\
fiber	& 53.32	& 39	& 24	& 0.01	& 0.01	& 64.97	& 22944	& 22881	& 0.88	& 2.86	& 21.85 \\
fixnet3	& 6.62	& 6	& 0	& 0	& 0	& 56.19	& 2925	& 2827	& 0.21	& 0.39	& 748.79 \\
fixnet4	& 4.79	& 6	& 0	& 0	& 0	& 13.02	& 2829	& 2731	& 0.24	& 0.4	& 171.82 \\
fixnet6	& 3.98	& 6	& 0	& 0	& 0	& 13.03	& 2502	& 2388	& 0.23	& 0.38	& 227.39 \\
khb05250	& 74.91	& 19	& 0	& 0	& 0	& 84.21	& 1435	& 1408	& 0.09	& 0.2	& 12.41 \\
l152lav	& 0	& 0	& 0	& 0	& 0.01	& 26.68	& 16774	& 16684	& 10.32	& 27.25	& 0.00 \\
lseu	& 55.94	& 12	& 7	& 0	& 0	& 56.59	& 567	& 562	& 0.01	& 0.02	& 1.16 \\
markshare1	& 0	& 6	& 3	& 0	& 0	& 0	& 124	& 110	& 0.01	& 0.01	& 0.00 \\
markshare2	& 0	& 7	& 3	& 0	& 0	& 0	& 173	& 147	& 0.02	& 0.03	& 0.00 \\
mas74	& 6.52	& 9	& 0	& 0	& 0	& 7.57	& 511	& 485	& 0.03	& 0.09	& 16.10 \\
mas76	& 6.36	& 9	& 1	& 0	& 0	& 7.7	& 433	& 412	& 0.03	& 0.08	& 21.07 \\
misc03	& 8.62	& 20	& 17	& 0	& 0	& 8.62	& 2239	& 2231	& 0.09	& 0.24	& 0.00 \\
misc06	& 26.17	& 8	& 0	& 0	& 0	& 26.17	& 8	& 0	& 0.01	& 0.06	& 0.00 \\
misc07	& 0	& 28	& 25	& 0	& 0	& 0.72	& 4177	& 4171	& 0.32	& 0.8	& 0.00 \\
mod008	& 20.1	& 4	& 1	& 0	& 0	& 20.27	& 96	& 90	& 0.02	& 0.04	& 0.85 \\
mod010	& 0	& 0	& 0	& 0	& 0.01	& 99.26	& 14211	& 14070	& 11.45	& 27.76	& 0.00 \\
mod011	& 11.44	& 8	& 1	& 0.14	& 0.14	& 32.81	& 855	& 171	& 6.47	& 7.85	& 186.80 \\
modglob	& 13.32	& 16	& 2	& 0	& 0	& 16.26	& 137	& 40	& 0.05	& 0.08	& 22.07 \\
p0033	& 12.6	& 5	& 1	& 0	& 0	& 57.04	& 174	& 168	& 0	& 0	& 352.70 \\
p0201	& 16.89	& 20	& 13	& 0	& 0	& 19.31	& 1933	& 1929	& 0.1	& 0.21	& 14.33 \\
p0282	& 3.47	& 24	& 17	& 0	& 0	& 6.2	& 2837	& 2829	& 0.05	& 0.15	& 78.67 \\
p0548	& 3.06	& 19	& 2	& 0	& 0.01	& 18.53	& 5584	& 5521	& 0.2	& 0.42	& 505.56 \\
p2756	& 0.21	& 7	& 1	& 0.01	& 0.01	& 0.56	& 908	& 884	& 0.16	& 0.3	& 166.67 \\
pk1	& 0	& 15	& 5	& 0	& 0	& 0	& 22	& 12	& 0	& 0.04	& 0.00 \\
pp08a	& 54.3	& 50	& 0	& 0	& 0	& 65.29	& 7825	& 7732	& 0.09	& 0.19	& 20.24 \\
pp08aCUTS	& 32.83	& 40	& 0	& 0	& 0	& 41.18	& 2579	& 2513	& 0.11	& 0.31	& 25.43 \\
qiu	& 0.33	& 36	& 24	& 0.16	& 0.16	& 0.33	& 36	& 24	& 0.16	& 0.57	& 0.00 \\
rentacar	& 0	& 2	& 0	& 0.02	& 0.03	& 0	& 7	& 5	& 0.13	& 0.51	& 0.00 \\
rgn	& 3.15	& 17	& 9	& 0	& 0	& 3.15	& 1341	& 1331	& 0.03	& 0.08	& 0.00 \\
set1ch	& 30.36	& 125	& 1	& 0.01	& 0.01	& 44.51	& 28347	& 28111	& 0.42	& 1.86	& 46.61 \\
stein27	& 0	& 21	& 17	& 0	& 0	& 0	& 1257	& 1254	& 0.02	& 0.03	& 0.00 \\
stein45	& 0	& 35	& 28	& 0	& 0	& 0	& 3584	& 3575	& 0.08	& 0.15	& 0.00 \\
swath	& 8.18	& 10	& 0	& 0.02	& 0.02	& 11.79	& 6917	& 6895	& 5.99	& 16.72	& 44.13 \\
vpm1	& 20.73	& 12	& 0	& 0	& 0	& 23.94	& 1073	& 1047	& 0.05	& 0.06	& 15.48 \\
vpm2	& 11.25	& 27	& 6	& 0	& 0	& 16.96	& 4345	& 4296	& 0.07	& 0.17	& 50.76 \\
\hline
Average	& 19.49	& 19.54	& 6.07	& 0.01	& 0.02	& 29.06	& 3903.2	& 3846.29	& 1.97	& 4.75	& 49.14 \\
\end{tabular}
\end{center}
\end{table}
\end{small}

\begin{small}
\begin{table}
\caption{GMI cuts + Triangle cuts, strengthened, 1 round}
\label{tab9.2}
\begin{center}
\footnotesize
\begin{tabular}{ l | r r r D D | r r r D D | r}	\cline{2-11}
  & \multicolumn{5}{c|}{GMI} & \multicolumn{5}{c|}{GMI + $T_A$, strengthened} &  \\\cline{2-11}
  & \multicolumn{1}{c}{Gap} & \multicolumn{1}{c}{Cuts} & \multicolumn{1}{c}{Cuts} & &  & \multicolumn{1}{c}{Gap} & \multicolumn{1}{c}{Cuts} & \multicolumn{1}{c}{Cuts} & & & \\
  & \multicolumn{1}{c}{closed} & \multicolumn{1}{c}{added} & \multicolumn{1}{c}{deleted} & &  & \multicolumn{1}{c}{closed} & \multicolumn{1}{c}{added} & \multicolumn{1}{c}{deleted} & & & Improvement\\
Instance  & \multicolumn{1}{c}{\%} & \multicolumn{1}{c}{\#} & \multicolumn{1}{c}{\#} & &  & \multicolumn{1}{c}{\%} & \multicolumn{1}{c}{\#} & \multicolumn{1}{c}{\#} & & & \multicolumn{1}{c}{\%}\\\hline
air03	& 100	& 36	& 10	& 0.05	& 0.12	& 100	& 5646	& 5569	& 17.59	& 42.04	& 0.00 \\
cap6000	& 41.65	& 2	& 1	& 0.02	& 0.02	& 41.65	& 364	& 363	& 0.66	& 2.4	& 0.00 \\
danoint	& 0.26	& 24	& 13	& 0.02	& 0.02	& 0.26	& 24	& 13	& 0.02	& 0.35	& 0.00 \\
dcmulti	& 45.75	& 49	& 12	& 0.01	& 0.01	& 45.75	& 119	& 86	& 0.01	& 0.17	& 0.00 \\
egout	& 21.84	& 16	& 0	& 0	& 0	& 60.9	& 2026	& 1990	& 0.02	& 0.04	& 178.85 \\
enigma	& 100	& 6	& 5	& 0	& 0	& 100	& 143	& 142	& 0	& 0.02	& 0.00 \\
fiber	& 53.32	& 39	& 24	& 0.01	& 0.01	& 59.77	& 9927	& 9877	& 0.36	& 1.16	& 12.10 \\
fixnet3	& 6.62	& 6	& 0	& 0	& 0	& 47.01	& 1571	& 1479	& 0.12	& 0.22	& 610.12 \\
fixnet4	& 4.79	& 6	& 0	& 0	& 0	& 12.49	& 1523	& 1426	& 0.13	& 0.23	& 160.75 \\
fixnet6	& 3.98	& 6	& 0	& 0	& 0.01	& 12.03	& 1352	& 1240	& 0.14	& 0.23	& 202.26 \\
khb05250	& 74.91	& 19	& 0	& 0	& 0	& 84.21	& 722	& 698	& 0.06	& 0.11	& 12.41 \\
l152lav	& 0	& 0	& 0	& 0	& 0	& 13.42	& 7568	& 7483	& 4.51	& 13.16	& 0.00 \\
lseu	& 55.94	& 12	& 7	& 0	& 0	& 55.94	& 291	& 286	& 0	& 0.02	& 0.00 \\
markshare1	& 0	& 6	& 3	& 0	& 0	& 0	& 66	& 56	& 0	& 0.01	& 0.00 \\
markshare2	& 0	& 7	& 3	& 0	& 0	& 0	& 91	& 69	& 0	& 0.02	& 0.00 \\
mas74	& 6.52	& 9	& 0	& 0	& 0	& 7.44	& 261	& 236	& 0.02	& 0.05	& 14.11 \\
mas76	& 6.36	& 9	& 1	& 0	& 0	& 7.12	& 225	& 203	& 0.01	& 0.04	& 11.95 \\
misc03	& 8.62	& 20	& 17	& 0	& 0	& 8.62	& 1031	& 1028	& 0.04	& 0.12	& 0.00 \\
misc06	& 26.17	& 8	& 0	& 0.01	& 0.01	& 26.17	& 8	& 0	& 0	& 0.04	& 0.00 \\
misc07	& 0	& 28	& 25	& 0	& 0.01	& 0	& 1933	& 1930	& 0.15	& 0.38	& 0.00 \\
mod008	& 20.1	& 4	& 1	& 0	& 0	& 20.11	& 48	& 43	& 0.01	& 0.02	& 0.05 \\
mod010	& 0	& 0	& 0	& 0	& 0.01	& 93.23	& 5873	& 5735	& 4.87	& 12.16	& 0.00 \\
mod011	& 11.44	& 8	& 1	& 0.13	& 0.14	& 32.53	& 464	& 103	& 3.09	& 3.78	& 184.35 \\
modglob	& 13.32	& 16	& 2	& 0	& 0	& 15.75	& 107	& 12	& 0.01	& 0.04	& 18.24 \\
p0033	& 12.6	& 5	& 1	& 0	& 0	& 57.04	& 85	& 80	& 0	& 0	& 352.70 \\
p0201	& 16.89	& 20	& 13	& 0	& 0	& 19.31	& 1235	& 1230	& 0.06	& 0.13	& 14.33 \\
p0282	& 3.47	& 24	& 17	& 0	& 0	& 5.38	& 1416	& 1409	& 0.03	& 0.08	& 55.04 \\
p0548	& 3.06	& 19	& 2	& 0	& 0.01	& 17.37	& 2958	& 2917	& 0.11	& 0.24	& 467.65 \\
p2756	& 0.21	& 7	& 1	& 0	& 0.01	& 0.56	& 546	& 523	& 0.08	& 0.16	& 166.67 \\
pk1	& 0	& 15	& 5	& 0	& 0	& 0	& 18	& 8	& 0	& 0.02	& 0.00 \\
pp08a	& 54.3	& 50	& 0	& 0	& 0	& 64.89	& 3862	& 3774	& 0.03	& 0.1	& 19.50 \\
pp08aCUTS	& 32.83	& 40	& 0	& 0.01	& 0.01	& 40.9	& 1132	& 1058	& 0.05	& 0.16	& 24.58 \\
qiu	& 0.33	& 36	& 24	& 0.15	& 0.16	& 0.33	& 36	& 24	& 0.16	& 0.4	& 0.00 \\
rentacar	& 0	& 2	& 0	& 0.02	& 0.03	& 0	& 4	& 2	& 0.06	& 0.27	& 0.00 \\
rgn	& 3.15	& 17	& 9	& 0	& 0	& 3.15	& 697	& 689	& 0.03	& 0.04	& 0.00 \\
set1ch	& 30.36	& 125	& 1	& 0	& 0	& 44.3	& 12412	& 12154	& 0.22	& 0.89	& 45.92 \\
stein27	& 0	& 21	& 17	& 0	& 0	& 0	& 846	& 843	& 0	& 0.01	& 0.00 \\
stein45	& 0	& 35	& 28	& 0	& 0	& 0	& 2395	& 2386	& 0.04	& 0.09	& 0.00 \\
swath	& 8.18	& 10	& 0	& 0.02	& 0.03	& 11.79	& 3358	& 3338	& 3.03	& 8.64	& 44.13 \\
vpm1	& 20.73	& 12	& 0	& 0	& 0	& 21.94	& 480	& 457	& 0.02	& 0.03	& 5.84 \\
vpm2	& 11.25	& 27	& 6	& 0	& 0	& 16.02	& 2133	& 2094	& 0.03	& 0.1	& 42.40 \\
\hline
Average	& 19.49	& 19.54	& 6.07	& 0.01	& 0.02	& 27.98	& 1829.17	& 1781.78	& 0.87	& 2.15	& 43.61 \\
\end{tabular}
\end{center}
\end{table}
\end{small}

\begin{small}
\begin{table}
\caption{GMI cuts + Conic cuts, strengthened, 1 round}
\label{tab9.3}
\begin{center}
\footnotesize
\begin{tabular}{ l | r r r D D | r r r D D | r}	\cline{2-11}
  & \multicolumn{5}{c|}{GMI} & \multicolumn{5}{c|}{GMI + $C_A$, strengthened} &  \\\cline{2-11}
  & \multicolumn{1}{c}{Gap} & \multicolumn{1}{c}{Cuts} & \multicolumn{1}{c}{Cuts} & &  & \multicolumn{1}{c}{Gap} & \multicolumn{1}{c}{Cuts} & \multicolumn{1}{c}{Cuts} & & & \\
  & \multicolumn{1}{c}{closed} & \multicolumn{1}{c}{added} & \multicolumn{1}{c}{deleted} & &  & \multicolumn{1}{c}{closed} & \multicolumn{1}{c}{added} & \multicolumn{1}{c}{deleted} & & & Improvement\\
Instance  & \multicolumn{1}{c}{\%} & \multicolumn{1}{c}{\#} & \multicolumn{1}{c}{\#} & &  & \multicolumn{1}{c}{\%} & \multicolumn{1}{c}{\#} & \multicolumn{1}{c}{\#} & & & \multicolumn{1}{c}{\%}\\\hline
air03	& 100	& 36	& 10	& 0.06	& 0.12	& 100	& 7134	& 7057	& 22.46	& 52.47	& 0.00 \\
cap6000	& 41.65	& 2	& 1	& 0.02	& 0.02	& 41.65	& 998	& 997	& 1.71	& 5.89	& 0.00 \\
danoint	& 0.26	& 24	& 13	& 0.02	& 0.02	& 0.26	& 24	& 13	& 0.01	& 0.2	& 0.00 \\
dcmulti	& 45.75	& 49	& 12	& 0.01	& 0.01	& 45.86	& 322	& 292	& 0.04	& 0.16	& 0.24 \\
egout	& 21.84	& 16	& 0	& 0	& 0	& 35.15	& 1451	& 1422	& 0.01	& 0.03	& 60.94 \\
enigma	& 100	& 6	& 5	& 0	& 0	& 100	& 204	& 203	& 0.02	& 0.02	& 0.00 \\
fiber	& 53.32	& 39	& 24	& 0.01	& 0.01	& 62.94	& 13056	& 12970	& 0.52	& 1.72	& 18.04 \\
fixnet3	& 6.62	& 6	& 0	& 0	& 0	& 49.3	& 1360	& 1257	& 0.08	& 0.16	& 644.71 \\
fixnet4	& 4.79	& 6	& 0	& 0	& 0	& 10.41	& 1312	& 1210	& 0.08	& 0.16	& 117.33 \\
fixnet6	& 3.98	& 6	& 0	& 0	& 0	& 11.5	& 1156	& 1050	& 0.09	& 0.16	& 188.94 \\
khb05250	& 74.91	& 19	& 0	& 0	& 0	& 77.59	& 732	& 704	& 0.05	& 0.1	& 3.58 \\
l152lav	& 0	& 0	& 0	& 0	& 0	& 26.68	& 9206	& 9118	& 5.6	& 13.67	& 0.00 \\
lseu	& 55.94	& 12	& 7	& 0	& 0	& 56.59	& 288	& 283	& 0.01	& 0.02	& 1.16 \\
markshare1	& 0	& 6	& 3	& 0	& 0	& 0	& 64	& 55	& 0.01	& 0.01	& 0.00 \\
markshare2	& 0	& 7	& 3	& 0	& 0	& 0	& 89	& 73	& 0.01	& 0.01	& 0.00 \\
mas74	& 6.52	& 9	& 0	& 0	& 0	& 7.15	& 259	& 237	& 0.03	& 0.05	& 9.66 \\
mas76	& 6.36	& 9	& 1	& 0	& 0	& 7.62	& 217	& 190	& 0.03	& 0.04	& 19.81 \\
misc03	& 8.62	& 20	& 17	& 0	& 0	& 8.62	& 1228	& 1220	& 0.06	& 0.13	& 0.00 \\
misc06	& 26.17	& 8	& 0	& 0	& 0	& 26.17	& 8	& 0	& 0	& 0.04	& 0.00 \\
misc07	& 0	& 28	& 25	& 0	& 0.01	& 0.72	& 2272	& 2266	& 0.16	& 0.41	& 0.00 \\
mod008	& 20.1	& 4	& 1	& 0	& 0	& 20.27	& 52	& 46	& 0.01	& 0.02	& 0.85 \\
mod010	& 0	& 0	& 0	& 0	& 0	& 97.91	& 8338	& 8198	& 6.75	& 15.37	& 0.00 \\
mod011	& 11.44	& 8	& 1	& 0.13	& 0.14	& 27.3	& 399	& 99	& 2.13	& 2.8	& 138.64 \\
modglob	& 13.32	& 16	& 2	& 0	& 0	& 13.94	& 46	& 2	& 0.01	& 0.03	& 4.65 \\
p0033	& 12.6	& 5	& 1	& 0	& 0	& 24.59	& 94	& 89	& 0.01	& 0.01	& 95.16 \\
p0201	& 16.89	& 20	& 13	& 0	& 0	& 16.89	& 718	& 711	& 0.04	& 0.09	& 0.00 \\
p0282	& 3.47	& 24	& 17	& 0	& 0	& 5.4	& 1445	& 1436	& 0.02	& 0.07	& 55.62 \\
p0548	& 3.06	& 19	& 2	& 0	& 0	& 6.53	& 2645	& 2599	& 0.08	& 0.18	& 113.40 \\
p2756	& 0.21	& 7	& 1	& 0	& 0	& 0.21	& 369	& 358	& 0.07	& 0.13	& 0.00 \\
pk1	& 0	& 15	& 5	& 0	& 0	& 0	& 19	& 9	& 0	& 0.02	& 0.00 \\
pp08a	& 54.3	& 50	& 0	& 0	& 0	& 57.01	& 4013	& 3927	& 0.04	& 0.1	& 4.99 \\
pp08aCUTS	& 32.83	& 40	& 0	& 0	& 0	& 34.23	& 1487	& 1409	& 0.06	& 0.16	& 4.26 \\
qiu	& 0.33	& 36	& 24	& 0.16	& 0.16	& 0.33	& 36	& 24	& 0.16	& 0.32	& 0.00 \\
rentacar	& 0	& 2	& 0	& 0.03	& 0.03	& 0	& 5	& 3	& 0.09	& 0.28	& 0.00 \\
rgn	& 3.15	& 17	& 9	& 0	& 0	& 3.15	& 661	& 650	& 0.01	& 0.04	& 0.00 \\
set1ch	& 30.36	& 125	& 1	& 0.01	& 0.01	& 37.43	& 16060	& 15780	& 0.24	& 1	& 23.29 \\
stein27	& 0	& 21	& 17	& 0	& 0	& 0	& 432	& 428	& 0.01	& 0.02	& 0.00 \\
stein45	& 0	& 35	& 28	& 0	& 0.01	& 0	& 1224	& 1214	& 0.03	& 0.05	& 0.00 \\
swath	& 8.18	& 10	& 0	& 0.02	& 0.02	& 9.89	& 3569	& 3537	& 3.08	& 7.84	& 20.90 \\
vpm1	& 20.73	& 12	& 0	& 0	& 0	& 22.73	& 605	& 575	& 0.02	& 0.03	& 9.65 \\
vpm2	& 11.25	& 27	& 6	& 0	& 0	& 15.25	& 2239	& 2199	& 0.04	& 0.09	& 35.56 \\
\hline
Average	& 19.49	& 19.54	& 6.07	& 0.01	& 0.02	& 25.88	& 2093.56	& 2046.59	& 1.07	& 2.54	& 32.83 \\
\end{tabular}
\end{center}
\end{table}
\end{small}

\begin{small}
\begin{table}
\caption{GMI cuts + Triangle cuts + Conic cuts, all strengthened, 5 rounds}
\label{tab9.4}
\begin{center}
\begin{tabular}{ c | c D D D D | c | D D D D c c}\cline{2-10}
      & ~~~~~~GMI~~~~~~    & --    & --    & --    & --    & GMI + $T_A$ + $C_A$, strengthened	& --	& --	& --	& --	& \\\cline{2-10}
      & \multicolumn{5}{c|}{~~~~~Gap~~~~~}       & \multicolumn{1}{c|}{Gap}                          & --    & --    & --    & --    &  \\
      & \multicolumn{5}{c|}{~~~~~Closed~~~~}    & \multicolumn{1}{c|}{Closed}                            & --    & --    & --    & --    & Improvement \\
round & \multicolumn{5}{c|}{~~~~~\%~~~~}        & \multicolumn{1}{c|}{\%}                                & --    & --    & --    & --    & \% \\      \hline
1	& 19.49	& 19.54	& 6.07	& 0.01	& 0.02	& 29.06	& 3903.20	& 3846.29	& 1.97	& 4.75	& 49.14	 \\
2	& 24.94	& 18.93	& 10.12	& 0.01	& 0.01	& 34.01	& 5104.07	& 5090.66	& 2.15	& 5.44	& 36.38	 \\
3	& 27.87	& 19.49	& 14.07	& 0.01	& 0.01	& 36.70	& 4746.00	& 4745.71	& 1.80	& 5.46	& 31.67	 \\
4	& 29.57	& 19.17	& 15.71	& 0.01	& 0.01	& 37.95	& 6355.02	& 6353.10	& 2.33	& 8.49	& 28.31	 \\
5	& 30.48	& 19.20	& 16.80	& 0.01	& 0.01	& 38.78	& 5753.29	& 5751.15	& 3.09	& 14.09	& 27.24	 \\
\end{tabular}
\end{center}
\end{table}
\end{small}

\end{document}